\documentclass[12pt,a4paper]{amsart}

\usepackage{amsmath,amssymb,amsthm}
\usepackage{mathtools}
\usepackage{enumitem}
\usepackage{hyperref}
\usepackage{geometry}
\usepackage{booktabs}
\newtheorem{theorem}{Theorem}[section]
\newtheorem{lemma}[theorem]{Lemma}
\newtheorem{proposition}[theorem]{Proposition}
\newtheorem{corollary}[theorem]{Corollary}
\theoremstyle{definition}
\newtheorem{definition}[theorem]{Definition}

\theoremstyle{remark}
\newtheorem{remark}[theorem]{Remark}

\newcommand{\SD}{\mathrm{SD}}
\newcommand{\Z}{\mathbb{Z}}
\DeclareMathOperator{\ord}{ord}

\begin{document}

\title[Zappa--Sz\'ep products of semidihedral groups]{On Zappa--Sz\'ep products of two semidihedral groups}

\author{Riccardo Aragona}
\address{Dipartimento di Ingegneria e Scienze dell'Informazione e Matematica,
Universit\`a degli Studi dell'Aquila, Via Vetoio, 67100 L'Aquila, Italy}
\email{riccardo.aragona@univaq.it}
\thanks{The author's ORCiD: 0000-0001-8834-4358}
\thanks{The author is a member of INdAM-GNSAGA (Italy).}
\thanks{The author declares that he has no conflict of interest.}
\subjclass[2020]{20D40, 20D60, 20F05, 20D15, 20E22}
\keywords{Zappa--Sz\'ep products, exact factorisations of finite groups, semidihedral group, polycyclic presentations, consistency conditions}
\date{\today}

\begin{abstract}
	Let $n, m \ge 4$.   We study the Zappa--Sz\'ep products
	$G = HK$ with $H = \langle x\rangle \rtimes \langle y\rangle \cong \SD_{2^n}$
	and $K = \langle z\rangle \rtimes \langle w\rangle \cong \SD_{2^m}$,
	together with the fixed generators $x, y, z, w$, and we describe the
	parameter tuples which encode the mixed commutators of these
	generators, according to the cores of $\langle x\rangle$ and
	$\langle z\rangle$ in~$G$.  First, when both $\langle x\rangle$ and
	$\langle z\rangle$ are normal in~$G$, we show that the admissible
	parameter tuples are exactly the solutions of an explicit system of
	six polynomial congruences.
	Second, when the cores $\langle x\rangle_G$ and $\langle z\rangle_G$
	are arbitrary subgroups of $\langle x\rangle$ and $\langle z\rangle$,
	under the simplifying assumption $[x, z] = 1$, and for generating
	systems for which all mixed commutators lie in
	$\langle x\rangle\langle z\rangle$, we obtain an analogous
	classification by twelve congruences together with two order
	conditions; this is the semidihedral counterpart of the Hu--Yu
	classification~\cite{HuYu2025} for dihedral groups $D_{2n}$
	and $D_{2m}$ with $n, m$ odd.   In contrast with that case, we further construct an explicit exact
	product with both cores non-trivial and $[x, z] \ne 1$, so that
	the semidihedral setting contains additional families of exact
	products.
\end{abstract}

\maketitle

\section{Introduction}

An expression $G = HK$ of a finite group $G$ as a product of proper
subgroups $H$ and $K$ is called a \emph{factorisation} of $G$.   When
$H\cap K = \{1\}$, the factorisation is called \emph{exact}, and in
this case we also say that $G$ is the \emph{Zappa--Sz\'ep} product of
$H$ and $K$~\cite{Szep1949,Zappa1940}.

The factorisation problem -- that is, the problem of determining all
factorisations of a given group -- is a longstanding question in
group theory, arising naturally from the work of Ore~\cite{Ore1937}.
Earlier contributions in this direction had already appeared in the
work of Miller~\cite{Miller1935}, who investigated conditions under
which a group can be expressed as the product of two subgroups.

The problem of describing the structure of a group factorised by two
prescribed subgroups has a long and rich history.   Several landmark
results were established by leading group theorists in the twentieth
century.   It\^o~\cite{Ito1955} proved that any product of two abelian
subgroups is metabelian, Douglas~\cite{Douglas1961} proved that
products of two cyclic groups are supersolvable, and the results of
Wielandt~\cite{Wielandt1951} and Kegel~\cite{Kegel1961} established that
products of two nilpotent subgroups are necessarily soluble.

Factorisations of almost simple groups, and the closely related
regular subgroups of primitive permutation groups, have been studied
extensively, see for instance~\cite{BurnessLi2021, HeringLiebeckSaxl1987,
LiWangXia2023, LPS1996, LiebeckPraegerSaxl2010, Xia2017}, and the
study of cyclic regular subgroups of primitive groups goes back to
Jones~\cite{Jones2002}.   These results are not used in the present
paper, which is concerned with the exact products of two given groups
of small rank.

In this paper we study the following variant of the factorisation
problem, namely, given two finite groups $H$ and $K$, to determine all
of their exact products, that is, all Zappa--Sz\'ep products of $H$
and $K$.

The exact products of two given groups have been determined in
several cases in which both groups have a cyclic subgroup of index at
most~$2$.   Hu and Yu~\cite{HuYu2025} classified the exact products
$G = HK$ with $H \cong D_{2n}$ and $K \cong D_{2m}$, for $n, m \ge 3$
odd, by integer parameters in suitable cyclic groups subject to
explicit congruences, and in their setting the cyclic subgroups of
index~$2$ of $H$ and $K$ commute.   The exact products of a dihedral
group and a cyclic group were studied in~\cite{HKK2024}, those of a
semidihedral group and a cyclic group in~\cite{Yu2025}, and those
of a generalised quaternion group with a cyclic group
in~\cite{DuLuoYu2024}.   In these last works one of the two
factors is cyclic.

The present paper initiates the study of the exact products of two
\emph{semidihedral} groups.   Recall that the semidihedral group of
order $2^n$, for $n \ge 4$, is
\begin{equation}\label{eq:SD-def}
  \SD_{2^n}
  = \langle x,y \mid x^{2^{n-1}}=y^2=1,\; x^y=x^{\alpha}\rangle,
  \qquad \alpha = 2^{n-2}-1.
\end{equation}
Throughout the paper we fix integers $n, m \ge 4$ and set
\[
  N := 2^{n-1}, \qquad M := 2^{m-1},
\]
so that $|\SD_{2^n}| = 2N$ and $|\SD_{2^m}| = 2M$, with
$\alpha = 2^{n-2}-1$ and $\beta = 2^{m-2}-1$ the
exponents of the corresponding semidihedral automorphisms.

We denote by \(G\) a Zappa--Sz\'ep product \(HK\) of \(H= \langle x\rangle \rtimes \langle y\rangle \cong \SD_{2^n}\) and \(K= \langle z\rangle \rtimes \langle w\rangle \cong \SD_{2^m}\). We parameterise the action of $H$ on $K$ and of
$K$ on $H$ in  $G = HK$ by writing the mixed
commutators in the form
\[
  [x,z] = x^{e_1} z^{e_2},\quad
  [z,y] = x^r z^a,\quad
  [x,w] = x^s z^b,\quad
  [y,w] = x^t z^c.
\]
We carry out a classification of Zappa--Sz\'ep products \(HK\) according to the cores $\langle x\rangle_G$ and $\langle z\rangle_G$ in~$G$, obtaining the
following two main results.
Here and throughout, by a classification we mean a description of
the parameter tuples which arise from exact products with fixed
generators $x, y, z, w$, in the precise sense of
Definition~\ref{def:distinguished} and Remark~\ref{rem:what-classified}
below.   The parameter tuples depend on the choice of the generators,
and we do not address the question of when two tuples give isomorphic
exact products.

\medskip\noindent\textbf{Theorem~A} (Theorem~\ref{thm:main}).
\emph{When $\langle x\rangle$ and $\langle z\rangle$ are normal in
$G$, that is, when $\langle x\rangle_G = \langle x\rangle$ and
$\langle z\rangle_G = \langle z\rangle$, the parameters
$e_1, e_2, r, b$ all vanish.   A tuple
$(a, s, t, c) \in \Z_M \times \Z_N \times \Z_N \times \Z_M$ is the
parameter tuple of such an exact product, with respect to suitable
generators $x, y, z, w$, if and only if it satisfies a system of six
explicit polynomial congruences, given in
equations~\eqref{eq:C1}--\eqref{eq:C6} below.   Moreover, the exact
product is determined up to isomorphism by its parameter tuple.
In other words, exact products with $\langle x\rangle$ and
$\langle z\rangle$ normal, together with a choice of generators
$x, y, z, w$, considered up to isomorphisms mapping generators to
generators, are in bijection with the tuples $(a, s, t, c)$
satisfying~\eqref{eq:C1}--\eqref{eq:C6}.}

\medskip\noindent\textbf{Theorem~B} (Theorem~\ref{thm:gen}).
\emph{%
Fix $n_1 \mid N$ and $m_1 \mid M$, and assume that the cores are
$\langle x\rangle_G = \langle x^{n_1}\rangle$ and
$\langle z\rangle_G = \langle z^{m_1}\rangle$.   Consider exact products with a
distinguished generating system in the sense of
Definition~\ref{def:distinguished}, under the simplifying assumption
$[x,z] = 1$.   A tuple $(r, a, s, b, t, c)$ is the parameter tuple of
such an exact product if and only if it satisfies twelve explicit
congruences together with the order conditions
$\ord_{\Z_N}(r) = m_1$ and $\ord_{\Z_M}(b) = n_1$.   Moreover, the exact
product is determined up to isomorphism by its parameter tuple, and the
same bijection as in Theorem~A holds.}

We stress that Theorem~B is not a classification of the exact
products with arbitrary cores.   It concerns a fixed pair
$(n_1, m_1)$, and it assumes both that $[x,z] = 1$ and that the
generators form a distinguished generating system.   By
Theorem~\ref{thm:xz-nontrivial-example}, the first assumption is
restrictive.

\medskip

Note that $n_1$ and $m_1$ are divisors of $N$ and $M$ which
describe the cores, and they should not be confused with the parameters
$n$ and $m$, which determine the orders $2^n$ and $2^m$ of $H$ and~$K$.
Theorem~B specialises to Theorem~A in the case $n_1 = m_1 = 1$.
The proofs combine elementary group-theoretic identities (for the
necessary condition) with the consistency theorem for polycyclic
presentations (for the sufficient condition,
see~\cite[Lemma 2.5]{Eick2000} and~\cite[p.~424]{Sims}).

Theorem~B is the semidihedral counterpart of the classification
of Hu and Yu~\cite{HuYu2025}, and we use the same generators
$x, y, z, w$.   However, in the dihedral case with $n$ and $m$ odd the
relation $[x,z] = 1$ is a consequence of the hypotheses, whereas here
it is an assumption.   The differences between the two cases come from
the arithmetic of the cyclic $2$-groups $\langle x\rangle$ and
$\langle z\rangle$.   First, for $n \ge 4$ the elements $u \in \Z_N$
with $u^2 \equiv 1 \mod{N}$ are $\pm 1$ and $\pm 1 + 2^{n-2}$, whereas
in a cyclic group of odd prime-power order, for instance of prime
order, only $\pm 1$ occur.   When
$\langle x\rangle$ is normal in $G$, the involution $w$ acts on
$\langle x\rangle$ as $x \mapsto x^u$ for one of these four values of
$u$, and similarly for the action of $y$ on $\langle z\rangle$.   The
two additional involutory automorphisms give rise to additional
families of exact products.   For instance, for $n = m = 4$, exactly
$72$ of the $144$ admissible tuples of Theorem~A have
$u = 1 + s \in \{3, 5\}$, that is, $w$ acts on $\langle x\rangle$
through one of the two additional automorphisms, see
Section~\ref{sec:examples}.   Second, in the odd case the proof that
$[x,z] = 1$ in~\cite{HuYu2025} ends by cancelling a factor $2$ in a
congruence modulo an odd number, see Remark~\ref{rem:dihedral}, and this
is not possible modulo powers of~$2$.   In fact the case $[x,z] \ne 1$
does occur, also with both cores non-trivial, as we show in
Theorem~\ref{thm:xz-nontrivial-example} and Subsection~\ref{sec:data44}.
This is why Theorem~B assumes $[x,z] = 1$, and why the general case
remains open.

The paper is organised as follows.
Section~\ref{sec:setup} fixes notation and records some key arithmetic
facts.   Section~\ref{sec:cores} establishes the lemma showing that
the normality hypothesis forces $[x,z] = 1$, $r = 0$ and $b = 0$.
In Section~\ref{sec:main-thm} we prove Theorem~\ref{thm:main}. In particular, the necessary condition
follows by direct expansion of group identities, and the sufficient
condition follows from the consistency theorem for polycyclic
presentations
(see~\cite[Lemma 2.5]{Eick2000} and~\cite[p.~424]{Sims}).   In Section~\ref{sec:examples} we exhibit
explicit counter-examples to the apparently natural claim that
$s = 0$ when $n = m$, with the simplest example already occurring in
the equal-rank case $n = m = 4$.   Section~\ref{sec:nontrivial}
generalises the classification to arbitrary cores under the assumption
$[x,z] = 1$.   In Subsection~\ref{sec:xz-nontrivial}, to show that the
assumption $[x,z] = 1$ in Theorem~B is genuinely restrictive, we
construct an explicit exact product with both cores non-trivial and
$[x,z] \ne 1$ (Theorem~\ref{thm:xz-nontrivial-example}), a
phenomenon that does not occur for dihedral groups $D_{2n}$ and
$D_{2m}$ with $n$ and $m$ odd (Remark~\ref{rem:dihedral}), and we
describe the computations for $n = m = 4$.   Finally,
Section~\ref{sec:open} collects some concluding remarks and a natural
open problem, and Appendix~\ref{app:gap-main} contains the GAP code
used in Section~\ref{sec:examples} and in
Theorem~\ref{thm:xz-nontrivial-example}.

\section{Notation and some arithmetic facts}\label{sec:setup}

Throughout the paper, all groups are finite.  We write $[g,h]=g^{-1}h^{-1}gh$ and $g^h=h^{-1}gh$,
so $g^h=g\cdot[g,h]$.   For a subgroup $U$ of $G$ we denote by
$U_G = \bigcap_{g \in G} U^g$ the core of $U$ in~$G$.

We shall use freely the elementary facts that, for $n \ge 4$ and
$\alpha = 2^{n-2} - 1$, we have $\alpha \equiv -1 \mod{2^{n-2}}$, so that
$\alpha$ is a unit of $\Z_{2^{n-1}}$, and
$\alpha^2 = 2^{2n-4} - 2^{n-1} + 1 \equiv 1 \mod{2^{n-1}}$, since
$2n - 4 \ge n - 1$.   The same holds for $\beta = 2^{m-2} - 1$ modulo
$2^{m-1}$.
The following arithmetic identities will also be used throughout,
namely the congruence
\begin{equation}\label{eq:alpha-mod4}
  \alpha = 2^{n-2}-1 \equiv 3 \mod{4},\quad n \ge 4,
\end{equation}
together with the identity $1+\alpha = 2^{n-2} \not\equiv 0 \mod{N}$.
The non-vanishing of $1+\alpha$ modulo $N$ is responsible for several
phenomena that have no analogue in the dihedral case, where
$\alpha = -1$ and so $1+\alpha = 0$, see for instance the comment after
Corollary~\ref{cor:parity}.

Let $n,m\ge 4$ and let $G=HK$ be a Zappa--Sz\'ep product of
$H= \langle x\rangle \rtimes \langle y\rangle \cong \SD_{2^n}$ and
$K= \langle z\rangle \rtimes \langle w\rangle \cong \SD_{2^m}$ as in
Theorem~A.   The conditions $\langle x\rangle_G = \langle x\rangle$
and $\langle z\rangle_G = \langle z\rangle$ are equivalent to
$\langle x\rangle, \langle z\rangle \trianglelefteq G$.

In this setting, every commutator of a generator of $H$ with a
generator of $K$ lies in $A := \langle x\rangle\langle z\rangle$.
Indeed, $y^{-1} z y \in \langle z\rangle$ by normality of
$\langle z\rangle$, so $[z,y] \in \langle z\rangle \subseteq A$;
similarly $[x,w] \in \langle x\rangle \subseteq A$.   For $[x,z]$
and $[y,w]$, observe that $A$ is normal in $G$, being the product of
two normal subgroups, and that $G/A$ is generated by the images of
$y$ and $w$, which have order at most $2$.   Hence $G/A$ is a
quotient of $\Z_2\times\Z_2$, in particular abelian, and so
$[x,z],[y,w] \in A$.   We may therefore write
\begin{equation}\label{eq:mixed}
	[x,z]=x^{e_1}z^{e_2},\quad
	[z,y]=x^r z^a,\quad
	[x,w]=x^s z^b,\quad
	[y,w]=x^t z^c,
\end{equation}
with $e_1,r,s,t\in\Z_N$ and $e_2,a,b,c\in\Z_M$.

In the more general setting of Theorem~B, in which the cores
$\langle x\rangle_G = \langle x^{n_1}\rangle$ and
$\langle z\rangle_G = \langle z^{m_1}\rangle$ need not be the whole
of $\langle x\rangle$ and $\langle z\rangle$, the subgroups
$\langle x\rangle$ and $\langle z\rangle$ need not be normal in~$G$
and the argument above does not apply.   We do not know
whether the mixed commutators always lie in
$\langle x\rangle\langle z\rangle$ in this case, and we therefore
restrict to generating systems for which this holds.   This leads to
the following definition, which makes precise the objects we
classify.

\begin{definition}\label{def:distinguished}
Let $G = HK$ be an exact product with $H \cong \SD_{2^n}$ and
$K \cong \SD_{2^m}$.   A \emph{distinguished generating system} of $G$
is a quadruple $(x, y, z, w)$ of elements of $G$ such that
$H = \langle x\rangle \rtimes \langle y\rangle$ and
$K = \langle z\rangle \rtimes \langle w\rangle$, with
\[
  x^N = y^2 = 1,\quad x^y = x^\alpha,\qquad
  z^M = w^2 = 1,\quad z^w = z^\beta,
\]
and such that the four commutators $[x,z]$, $[z,y]$, $[x,w]$ and
$[y,w]$ lie in the subset $\langle x\rangle\langle z\rangle$ of~$G$.
Since $\langle x\rangle \cap \langle z\rangle \subseteq H \cap K = \{1\}$,
every element of $\langle x\rangle\langle z\rangle$ can be written
uniquely as $x^i z^j$ with $i \in \Z_N$ and $j \in \Z_M$.   Hence there
is a unique tuple
\[
  (e_1, e_2, r, a, s, b, t, c) \in
  \Z_N \times \Z_M \times \Z_N \times \Z_M \times \Z_N \times \Z_M
  \times \Z_N \times \Z_M
\]
such that~\eqref{eq:mixed} holds, which we call the \emph{parameter
tuple} of $G$ with respect to $(x, y, z, w)$.
\end{definition}

By the argument above, when $\langle x\rangle$ and
$\langle z\rangle$ are normal in~$G$, every quadruple $(x, y, z, w)$
satisfying the first two conditions of
Definition~\ref{def:distinguished} is a distinguished generating
system.

Since $g^h = g\,[g,h]$ and $[h,g] = [g,h]^{-1}$, the parameter
tuple determines all the conjugates of a generator of one factor by a
generator of the other factor.   Namely, besides $x^y = x^\alpha$ and
$z^w = z^\beta$, we have
\begin{equation}\label{eq:all-conj}
\begin{aligned}
  x^z &= x^{1+e_1} z^{e_2}, & x^w &= x^{1+s} z^{b}, &
  y^z &= y\, z^{-a} x^{-r}, & y^w &= y\, x^{t} z^{c},\\
  z^x &= z^{1-e_2} x^{-e_1}, & z^y &= z\, x^{r} z^{a}, &
  w^x &= w\, z^{-b} x^{-s}, & w^y &= w\, z^{-c} x^{-t}.
\end{aligned}
\end{equation}
In particular, the mixed conjugations of an exact product with a
distinguished generating system are exactly those given
by~\eqref{eq:all-conj} for some parameter tuple, and the problem is to
determine which tuples occur.   When $[x,z] = 1$ these formulas
simplify, for instance $z^y = x^r z^{1+a}$.

\begin{remark}\label{rem:what-classified}
We make precise what is proved in Theorems~\ref{thm:main}
and~\ref{thm:gen}, distinguishing three statements which may be meant
by a classification.
\begin{enumerate}[label=\normalfont(\roman*)]
  \item Every exact product $G$ with a distinguished generating system
    $(x, y, z, w)$, satisfying the hypotheses of the theorem, has a
    parameter tuple satisfying the stated conditions, and $G$ is
    isomorphic to the group defined by the corresponding presentation,
    via an isomorphism which maps generators to generators.
  \item Every tuple satisfying the stated conditions is the parameter
    tuple of some exact product with respect to some distinguished
    generating system.
  \item The theorems do not address the question of when two parameter
    tuples give isomorphic exact products, either as abstract groups or
    as factorised groups, that is, via an isomorphism mapping the two
    factors onto the two factors.
\end{enumerate}
Equivalently, (i) and (ii) say that the map which associates with an
exact product $G$ and a distinguished generating system $(x, y, z, w)$
of $G$ its parameter tuple induces a bijection between the following
two sets.   The first is the set of such pairs $(G, (x, y, z, w))$,
satisfying the hypotheses of the theorem, up to isomorphisms mapping
$x, y, z, w$ to the corresponding generators.   The second is the set
of tuples satisfying the conditions of the theorem.   This is the only
sense in which the word bijection is used in this paper.
As for~(iii), the parameter tuple depends on the choice of the generators.
For instance, assume $[x,z] = 1$, and let $k \in \Z_N^\times$ and
$l \in \Z_M^\times$.   Then $(x^k, y, z^l, w)$ is again a distinguished
generating system of the same exact product, with the same factors $H$
and~$K$, and a direct computation shows that its parameter tuple is
obtained from $(0, 0, r, a, s, b, t, c)$ by replacing
\[
  (r, a, s, b, t, c) \quad\text{with}\quad
  (l k^{-1} r,\; a,\; s,\; k l^{-1} b,\; k^{-1} t,\; l^{-1} c).
\]
Thus tuples in the same orbit of this action of
$\Z_N^\times \times \Z_M^\times$ give the same factorised group.   As an
illustration, for $n = m = 4$ the $880$ tuples of Theorem~\ref{thm:gen}
described in Subsection~\ref{sec:data44} form only $503$ orbits under this
action.   Further
identifications arise by replacing $y$ and $w$ with other involutions
of $H$ and $K$, and from automorphisms of $G$ that do not preserve the
factors.   A classification up to isomorphism is beyond the scope of
the present paper.
\end{remark}

The following elementary fact is a key step in classifying the
parameters $a$ and $s$.   For $n \ge 4$, the elements $u \in \Z_N$ with
$u^2 \equiv 1 \mod{N}$ are $u \equiv \pm 1 \mod{2^{n-2}}$, that is,
$u \in \{1,\, 1 + 2^{n-2},\, -1,\, -1 + 2^{n-2}\}$ modulo~$N$.   Writing
$u = 1 + s$, the corresponding values of $s$ are
\begin{equation}\label{eq:s-set}
  s \in \{0,\, 2^{n-2} - 2,\, 2^{n-2},\, 2^{n-1} - 2\} \mod{N},
\end{equation}
and they are all even, since each $u$ is odd.   The analogous
statement holds in $\Z_M$ for~$a$.

\begin{remark}\label{rem:missing-solution}
	The values $s = 2^{n-2}-2$ and $s = 2^{n-2}$, corresponding to
	$u = -1 + 2^{n-2}$ and $u = 1 + 2^{n-2}$, have no analogue for
	cyclic groups of odd prime-power order, where only $u = \pm 1$
	occur.   For $n = 4$ they give $s = 2$ and $s = 4$.   Together with
	$s = 2^{n-1} - 2$, corresponding to $u = -1$, they give exact
	products with $s \ne 0$ also when $n = m$, see
	Section~\ref{sec:examples}.
\end{remark}

\section{The normality hypothesis}\label{sec:cores}

The following lemma shows that the normality hypothesis, although apparently mild, is very strong, since it eliminates four of the eight parameters in~\eqref{eq:mixed}.

\begin{lemma}\label{lem:trivial-cores}
Under the normality hypothesis
$\langle x\rangle\trianglelefteq G$ and $\langle z\rangle\trianglelefteq G$,
\[
  e_1 = 0\mod N,\quad
  e_2 = 0\mod M,\quad
  r = 0\mod N,\quad
  b = 0\mod M.
\]
Consequently $A = \langle x\rangle \times \langle z\rangle$ is an
abelian subgroup of $G$ isomorphic to $\Z_N \times \Z_M$, and the
conjugation formulae~\eqref{eq:mixed} reduce to
\begin{equation}\label{eq:reduced-conj}
  [x,z]=1,\qquad
  z^y = z^{1+a},\qquad
  x^w = x^{1+s},\qquad
  y^w = y\,x^t z^c.
\end{equation}
\end{lemma}

\begin{proof}
($e_1 = e_2 = 0$).  Since $\langle x\rangle\trianglelefteq G$,
$z^{-1} x z \in \langle x\rangle$, hence
\[
  [x,z] = x^{-1}(z^{-1} x z) \in \langle x\rangle.
\]
Since $\langle z\rangle\trianglelefteq G$,
$x^{-1} z^{-1} x = (x^{-1} z x)^{-1} \in \langle z\rangle$, hence
\[
  [x,z] = (x^{-1} z^{-1} x) z \in \langle z\rangle.
\]
Therefore $[x,z]\in \langle x\rangle \cap \langle z\rangle$.
Since $\langle x\rangle \le H$ and $\langle z\rangle \le K$, we have
$\langle x\rangle \cap \langle z\rangle \subseteq H \cap K = \{1\}$, a
fact that we shall use repeatedly.   Hence $[x,z] = 1$, that is,
$e_1 = 0$ and $e_2 = 0$.

 ($r = 0$).  Since $\langle z\rangle\trianglelefteq G$,
$z^y = y^{-1} z y \in \langle z\rangle$.  From~\eqref{eq:mixed},
\[
  z^y = z \cdot [z, y] = z \cdot x^r z^a = x^r z^{1+a}
\]
since $[x,z]=1$.   Hence $x^r = z^y z^{-(1+a)} \in \langle z\rangle$,
so that $x^r \in \langle x\rangle \cap \langle z\rangle = \{1\}$.   Since
$\langle x\rangle$ has order~$N$, $r\equiv 0\mod{N}$.

 ($b = 0$).  Symmetrically,
$\langle x\rangle\trianglelefteq G$ implies
$x^w = x \cdot [x, w] = x^{1+s} z^b \in \langle x\rangle$, hence
$z^b = x^{-(1+s)} x^w \in \langle x\rangle \cap \langle z\rangle = \{1\}$,
and $b \equiv 0\mod M$.

(Structure of $A$).  Since $[x, z] = 1$ and
$\langle x\rangle \cap \langle z\rangle \subseteq H \cap K = \{1\}$,
$A = \langle x\rangle \langle z\rangle \cong \Z_N \times \Z_M$.

The conjugation formulae~\eqref{eq:reduced-conj} follow by
substituting $r = b = 0$ into~\eqref{eq:mixed}.
\end{proof}

\section{Proof of the main theorem}\label{sec:main-thm}

We now prove our first main theorem.

\begin{theorem}\label{thm:main}
Let $n,m\ge 4$, with $\alpha=2^{n-2}-1$, $\beta=2^{m-2}-1$,
$N := 2^{n-1}$, $M := 2^{m-1}$.  Set
\begin{align*}
  H &= \langle x,y \mid x^N=y^2=1,\; x^y=x^\alpha\rangle \cong \SD_{2^n},\\
  K &= \langle z,w \mid z^M=w^2=1,\; z^w=z^\beta\rangle \cong \SD_{2^m}.
\end{align*}
Let $(a, s, t, c) \in \Z_M\times\Z_N\times\Z_N\times\Z_M$, and consider
the polycyclic presentation
\begin{equation}\label{eq:reduced-pres}
  G = \left\langle x, y, z, w \,\middle|\,
  \begin{array}{l}
    x^N = y^2 = z^M = w^2 = 1,\ [x,z] = 1, \\
    x^y = x^\alpha,\ z^w = z^\beta, \\
    {[z,y]} = z^a,\ {[x,w]} = x^s,\ {[y,w]} = x^t z^c
  \end{array}
  \right\rangle
\end{equation}
Then the presentation~\eqref{eq:reduced-pres} is consistent, that
is, it defines a group $G$ of order exactly $4NM$, if and only if
$(a, s, t, c)$ satisfies the six congruences
\begin{align}
  (1+a)^2 &\equiv 1 \mod{M}, \tag{C1}\label{eq:C1}\\
  (1+s)^2 &\equiv 1 \mod{N}, \tag{C2}\label{eq:C2}\\
  t\,(2 + s) &\equiv 0 \mod{N}, \tag{C3}\label{eq:C3}\\
  c\,(1+\beta) &\equiv 0 \mod{M}, \tag{C4}\label{eq:C4}\\
  t\,(1 + \alpha) &\equiv 0 \mod{N}, \tag{C5}\label{eq:C5}\\
  c\,(2 + a) &\equiv 0 \mod{M}. \tag{C6}\label{eq:C6}
\end{align}
In this case $G = HK$ is an exact product with
$H = \langle x, y\rangle \cong \SD_{2^n}$,
$K = \langle z, w\rangle \cong \SD_{2^m}$ and $\langle x\rangle$,
$\langle z\rangle$ normal in~$G$.
Conversely, every exact product $G = HK$ with
$H \cong \SD_{2^n}$, $K \cong \SD_{2^m}$ and
$\langle x\rangle, \langle z\rangle$ normal in $G$ admits a
presentation of the form~\eqref{eq:reduced-pres} with parameters
satisfying \eqref{eq:C1}--\eqref{eq:C6}.
More precisely, for every choice of generators $x, y, z, w$ as in
Definition~\ref{def:distinguished}, the parameter tuple of $G$ is of
the form $(0, 0, 0, a, s, 0, t, c)$ with $(a, s, t, c)$ satisfying
\eqref{eq:C1}--\eqref{eq:C6}, and $G$ is isomorphic to the group
defined by~\eqref{eq:reduced-pres} via an isomorphism mapping generators
to generators.   Different tuples may give isomorphic exact products,
see Remark~\ref{rem:what-classified}.
\end{theorem}

\subsection{Necessary condition}

We prove necessity by expanding four elementary group-theoretic
identities in the reduced conjugation rules~\eqref{eq:reduced-conj}.

\paragraph{Identity $(\mathcal{I}_1)$: $(z^y)^y = z$, obtained from $y^2=1$.} Notice that
\[
  (z^y)^y = (z^{1+a})^y = z^{(1+a)^2}.
\]
Setting equal to $z$ we have $(1+a)^2\equiv 1\mod{M}$, which is \eqref{eq:C1}.

\paragraph{Identity $(\mathcal{I}_2)$: $(x^w)^w = x$, obtained from $w^2=1$.} Notice that
\[
  (x^w)^w = (x^{1+s})^w = x^{(1+s)^2}.
\]
Setting equal to $x$ we have $(1+s)^2\equiv 1\mod{N}$, which is \eqref{eq:C2}.

\paragraph{Identity $(\mathcal{I}_3)$: $(y^w)^w = y$, obtained from $w^2=1$.}
Notice that
\begin{align*}
  (y^w)^w &= (y\,x^t z^c)^w = y^w \cdot (x^t z^c)^w
        = (y\,x^t z^c) \cdot (x^{t(1+s)}\,z^{c\beta})\\
        &= y \cdot x^{t(2+s)}\,z^{c(1+\beta)},
\end{align*}
where in the last step we used $[x,z]=1$.
Setting equal to $y$ gives \eqref{eq:C3} and \eqref{eq:C4}.

\paragraph{Identity $(\mathcal{I}_4)$: $(y^w)^2 = 1$, obtained from $y^2=1$.}
 Notice that
\begin{align*}
  (y^w)^2 &= (y\,x^t z^c)(y\,x^t z^c)
        = y \cdot (y\cdot y^{-1}\cdot x^t z^c\cdot y) \cdot x^t z^c\\
        &= y^2 \cdot (x^t z^c)^y \cdot x^t z^c
        = (x^{t\alpha}\,z^{c(1+a)})(x^t z^c)\\
        &= x^{t(1+\alpha)}\,z^{c(2+a)}.
\end{align*}
Setting equal to $1$ gives \eqref{eq:C5} and \eqref{eq:C6}. \qed

\begin{corollary}\label{cor:parity}
The parameters $a, s, t, c$ are all even.
\end{corollary}

\begin{proof}
\eqref{eq:C1} forces $1+a$ odd, hence $a$ even; symmetrically $s$ even
from \eqref{eq:C2}.  From \eqref{eq:C5}, $t\cdot(1+\alpha) = t\cdot 2^{n-2}\equiv 0\mod{N}$,
i.e.\ $2^{n-2}\cdot t\equiv 0\mod{2^{n-1}}$, so  $t$ is even.
Symmetrically $c$ is even from \eqref{eq:C4}.
\end{proof}

The parity of $t$ is a first instance of the phenomena due to the
non-vanishing of $1+\alpha$ modulo $N$ mentioned in
Section~\ref{sec:setup}.   In the dihedral case, where $\alpha = -1$,
the congruence corresponding to~\eqref{eq:C5} is trivially satisfied,
and it imposes no condition on~$t$.

\subsection{Sufficient condition via consistency theorem}\label{sec:cons-thm}

We view~\eqref{eq:reduced-pres} as a polycyclic presentation, also
called PC-presentation
(see~\cite[Section~2.3]{Eick2000} or~\cite[Definition~8.7]{HoltEickOBrien2005}),
and verify that it satisfies all the relevant consistency conditions.

Order the generators as
\[
  g_1 := w,\quad g_2 := y,\quad g_3 := z,\quad g_4 := x,
\]
with relative orders $\rho_1 = \rho_2 = 2$, $\rho_3 = M$, $\rho_4 = N$.  This
ordering corresponds to the subnormal series
\[
  G \,=\, G_0 \;\rhd\;
  G_1 := \langle x, z, y\rangle \;\rhd\;
  G_2 := \langle x, z\rangle = A \;\rhd\;
  G_3 := \langle x\rangle \;\rhd\;
  G_4 := \{1\},
\]
with cyclic quotients of orders $2, 2, M, N$ from top to bottom.
Using the relations, every element of the group defined
by~\eqref{eq:reduced-pres} can be written as a \emph{normal word}
$g_1^{\lambda_1} g_2^{\lambda_2} g_3^{\lambda_3} g_4^{\lambda_4}
= w^{\lambda_1} y^{\lambda_2} z^{\lambda_3} x^{\lambda_4}$, with $0\le \lambda_i < \rho_i$.   The
uniqueness of this normal form is not automatic, and it is exactly what
the consistency theorem below allows us to decide.

The conjugate relations, for $i < j$ and with the right-hand side
written as a normal word in the generators of index greater than $i$,
are
\begin{equation}\label{eq:conj-relations}
\begin{aligned}
  g_2^{g_1} &= y^w = y\,x^t z^c = g_2 \, g_3^c \, g_4^t,\\
  g_3^{g_1} &= z^w = z^\beta = g_3^\beta,\\
  g_4^{g_1} &= x^w = x^{1+s} = g_4^{1+s},\\
  g_3^{g_2} &= z^y = z^{1+a} = g_3^{1+a},\\
  g_4^{g_2} &= x^y = x^\alpha = g_4^\alpha,\\
  g_4^{g_3} &= x^z = x = g_4 \quad (\text{since } [x,z]=1).
\end{aligned}
\end{equation}
The power relations are $g_i^{\rho_i} = 1$ for all $i$.

We now state precisely the version of the consistency theorem that
we use.   Let $g_1, \ldots, g_\ell$ be generators and let
$\rho_1, \ldots, \rho_\ell \ge 2$ be integers.   For $1 \le i \le \ell$, a
\emph{normal word} in $g_i, \ldots, g_\ell$ is a word
$g_i^{\lambda_i} \cdots g_\ell^{\lambda_\ell}$ with $0 \le \lambda_k < \rho_k$.   Consider a
presentation with relations
\begin{equation}\label{eq:pc-general}
  g_i^{\rho_i} = \pi_i \quad (1 \le i \le \ell),
  \qquad
  g_j^{g_i} = \gamma_{ij} \quad (1 \le i < j \le \ell),
\end{equation}
where $\pi_i$ and $\gamma_{ij}$ are normal words in $g_{i+1}, \ldots, g_\ell$.
Using these relations, every word in the generators can be rewritten,
by a finite sequence of applications of the relations called
\emph{collection}, into a normal word in $g_1, \ldots, g_\ell$.   In
particular, the group defined by the presentation has order at most
$\rho_1 \cdots \rho_\ell$.

\begin{lemma}[{\cite[Lemma 2.5]{Eick2000}, \cite[p.~424]{Sims}}]\label{lem:consistency}
The group defined by~\eqref{eq:pc-general} has order exactly
$\rho_1 \cdots \rho_\ell$, that is, distinct normal words represent distinct
elements, if and only if, for each of the following words, the two
indicated ways of collecting it give the same normal word.
\begin{enumerate}[label=\normalfont(\arabic*)]
  \item $g_k g_j g_i$ for $i < j < k$, collected as $g_k(g_j g_i)$ and
    as $(g_k g_j) g_i$.
  \item $g_j^{\rho_j} g_i$ for $i < j$, collected as $(g_j^{\rho_j}) g_i$
    and as $g_j^{\rho_j - 1}(g_j g_i)$.
  \item $g_j g_i^{\rho_i}$ for $i < j$, collected as $g_j (g_i^{\rho_i})$
    and as $(g_j g_i) g_i^{\rho_i - 1}$.
  \item $g_i^{\rho_i + 1}$ for $1 \le i \le \ell$, collected as
    $g_i (g_i^{\rho_i})$ and as $(g_i^{\rho_i}) g_i$.
\end{enumerate}
In each case the bracketed subword is rewritten first, by means of the
corresponding relation in~\eqref{eq:pc-general}, and the resulting word
is then collected to a normal word.
\end{lemma}

The four families above are the first four families
of~\cite[Lemma 2.5(c)]{Eick2000}.   The fifth family listed in~\cite[Lemma 2.5(c)]{Eick2000}
involves the inverses of the generators of infinite relative order, and
it does not occur here, since all relative orders $\rho_i$ are finite.

We check the conditions of Lemma~\ref{lem:consistency} in
decreasing order of the smallest index~$i$ involved.   When the
conditions involving only $g_{i+1}, \ldots, g_\ell$ have been verified,
Lemma~\ref{lem:consistency} applied to the sub-presentation on these
generators shows that it defines a group of order
$\rho_{i+1} \cdots \rho_\ell$ in which normal words are unique.   Hence words
in $g_{i+1}, \ldots, g_\ell$ can be multiplied, collected and
conjugated inside this group, without any ambiguity.   For such a word
$v$ we denote by $v^{g_i}$ the word obtained by replacing every letter
$g_j$ of $v$ with $\gamma_{ij}$, computed in this group.   In the case of
trivial power relations, the conditions of Lemma~\ref{lem:consistency}
then take the following form.

\begin{lemma}\label{lem:reformulation}
Assume that $\pi_i = 1$ for every~$i$, and that the sub-presentation
on $g_{i+1}, \ldots, g_\ell$ is consistent.   Then the conditions of
Lemma~\ref{lem:consistency} with smallest index $i$ are equivalent to
the following identities between elements of the group defined by the
sub-presentation on $g_{i+1}, \ldots, g_\ell$.
\begin{enumerate}[label=\normalfont(\arabic*)]
  \item[\textup{(C)}] $(\gamma_{jk})^{g_i} = (\gamma_{ik})^{\gamma_{ij}}$ for
    $i < j < k$, that is, $(g_k^{g_j})^{g_i} = (g_k^{g_i})^{g_j^{g_i}}$.
  \item[\textup{(W)}] $(\gamma_{ij})^{\rho_j} = 1$ for $i < j$, that is,
    $(g_j^{g_i})^{\rho_j} = 1$.
  \item[\textup{(P)}] Applying $\rho_i$ times the map $v \mapsto v^{g_i}$
    to $g_j$ gives $g_j$, for $i < j$.
\end{enumerate}
Moreover, the conditions of family~\textup{(4)} hold automatically.
\end{lemma}

\begin{proof}
For family~(4), rewriting $g_i^{\rho_i} = \pi_i = 1$ turns both
$g_i(g_i^{\rho_i})$ and $(g_i^{\rho_i})g_i$ into the normal word $g_i$.
For family~(3), the first way gives $g_j \pi_i = g_j$.   For the second
way, $(g_j g_i) g_i^{\rho_i-1} = g_i \gamma_{ij} g_i^{\rho_i-1}$, and moving the
remaining $\rho_i - 1$ letters $g_i$ to the left by means of
$v g_i = g_i v^{g_i}$ gives $g_i^{\rho_i} v'$, where $v'$ is obtained from
$g_j$ by applying $\rho_i$ times the map $v \mapsto v^{g_i}$.   Rewriting
$g_i^{\rho_i} = 1$ gives $v'$, and the two normal words agree if and only
if $v' = g_j$, which is~(P).   For family~(2), the first way gives
$\pi_j g_i = g_i$, while the second gives
$g_j^{\rho_j-1} g_i \gamma_{ij} = g_i (\gamma_{ij})^{\rho_j-1} \gamma_{ij}
= g_i (\gamma_{ij})^{\rho_j}$, so the two agree if and only if~(W) holds.   For
family~(1), the first way gives
$g_k g_i \gamma_{ij} = g_i \gamma_{ik} \gamma_{ij}$, while the second gives
$g_j \gamma_{jk} g_i = g_j g_i (\gamma_{jk})^{g_i}
= g_i \gamma_{ij} (\gamma_{jk})^{g_i}$.   By the uniqueness of normal words in
the sub-presentation on $g_{i+1}, \ldots, g_\ell$, the two agree if and
only if $\gamma_{ik} \gamma_{ij} = \gamma_{ij} (\gamma_{jk})^{g_i}$ in the group it
defines, which is~(C).
\end{proof}

We now list all the conditions for the
presentation~\eqref{eq:reduced-pres}, with $g_1 = w$, $g_2 = y$,
$g_3 = z$, $g_4 = x$, relative orders
$(\rho_1, \rho_2, \rho_3, \rho_4) = (2, 2, M, N)$, trivial power
relations, and conjugate relations~\eqref{eq:conj-relations}.   There
are six conditions of type~(W), six of type~(P), four of type~(C) and
four of type~(4), and the last four hold automatically by
Lemma~\ref{lem:reformulation}.   The subscripts refer to the
generators $g_i$, so that, for instance, $(\mathrm{P})_{1,2}$ is the
condition obtained by conjugating $g_2 = y$ twice by $g_1 = w$, which
reads $(y^w)^w = y$, and $(\mathrm{C})_{1,2,3}$ compares
$(z^y)^w$ with $(z^w)^{y^w}$.

\emph{Smallest index $i = 3$.}   The sub-presentation on $z, x$
is $x^N = z^M = 1$, $x^z = x$.   Condition $(\mathrm{W})_{3,4}$ reads
$(x^z)^N = x^N = 1$, and condition $(\mathrm{P})_{3,4}$ holds since
$x^z = x$.   Hence this sub-presentation defines
$A = \langle z, x\rangle \cong \Z_M \times \Z_N$, and from now on words
in $z$ and $x$ are computed in this abelian group.   In particular,
conjugation by $y$ and by $w$ acts on $A$ by
\[
  (x^i z^j)^y = x^{\alpha i} z^{(1+a) j},
  \qquad
  (x^i z^j)^w = x^{(1+s) i} z^{\beta j}.
\]

\emph{Smallest index $i = 2$.}   We have
\begin{align*}
  (\mathrm{W})_{2,3}&\colon\quad (z^{1+a})^M = z^{(1+a)M} = 1,
  &(\mathrm{W})_{2,4}&\colon\quad (x^\alpha)^N = x^{\alpha N} = 1,\\
  (\mathrm{P})_{2,3}&\colon\quad z \mapsto z^{1+a} \mapsto z^{(1+a)^2},
  &(\mathrm{P})_{2,4}&\colon\quad x \mapsto x^\alpha \mapsto x^{\alpha^2},
\end{align*}
and $(\mathrm{C})_{2,3,4}$ reads $(x^z)^y = x^\alpha$ on the left and
$(x^y)^{z^y} = (x^\alpha)^{z^{1+a}} = x^\alpha$ on the right.   The two
conditions of type~(W) and $(\mathrm{C})_{2,3,4}$ hold identically,
$(\mathrm{P})_{2,4}$ holds since $\alpha^2 \equiv 1 \mod{N}$, and
$(\mathrm{P})_{2,3}$ is equivalent to~\eqref{eq:C1}.   Hence, under
\eqref{eq:C1}, the sub-presentation on $y, z, x$ defines a group
$G_1$ of order $2MN$, containing $A$ as a normal subgroup of index~$2$.

\emph{Smallest index $i = 1$.}   The words $\gamma_{1j}$ are
$y^w = y\,x^t z^c$, $z^w = z^\beta$ and $x^w = x^{1+s}$.   The
conditions of type~(W) are
\begin{align*}
  (\mathrm{W})_{1,2}&\colon\quad
  (y\,x^t z^c)^2 = y^2\,(x^t z^c)^y\,x^t z^c
  = x^{\alpha t} z^{(1+a)c}\,x^t z^c
  = x^{(1+\alpha) t}\, z^{(2+a) c} = 1,\\
  (\mathrm{W})_{1,3}&\colon\quad (z^\beta)^M = 1,
  \qquad
  (\mathrm{W})_{1,4}\colon\quad (x^{1+s})^N = 1,
\end{align*}
where $(\mathrm{W})_{1,2}$ is computed in $G_1$.   The conditions of
type~(P) are
\begin{align*}
  (\mathrm{P})_{1,2}&\colon\quad
  y \mapsto y\,x^t z^c \mapsto (y\,x^t z^c)^w
  = y\,x^t z^c\,(x^t z^c)^w
  = y\,x^{(2+s)t}\, z^{(1+\beta)c},\\
  (\mathrm{P})_{1,3}&\colon\quad z \mapsto z^\beta \mapsto z^{\beta^2},
  \qquad
  (\mathrm{P})_{1,4}\colon\quad x \mapsto x^{1+s} \mapsto x^{(1+s)^2}.
\end{align*}
Finally, the conditions of type~(C) are
\begin{align*}
  (\mathrm{C})_{1,2,3}&\colon\quad
  (z^y)^w = (z^{1+a})^w = z^{\beta(1+a)}
  \quad\text{and}\quad
  (z^w)^{y^w} = \bigl((z^\beta)^y\bigr)^{x^t z^c} = z^{\beta(1+a)},\\
  (\mathrm{C})_{1,2,4}&\colon\quad
  (x^y)^w = (x^\alpha)^w = x^{\alpha(1+s)}
  \quad\text{and}\quad
  (x^w)^{y^w} = \bigl((x^{1+s})^y\bigr)^{x^t z^c} = x^{\alpha(1+s)},\\
  (\mathrm{C})_{1,3,4}&\colon\quad
  (x^z)^w = x^{1+s}
  \quad\text{and}\quad
  (x^w)^{z^w} = (x^{1+s})^{z^\beta} = x^{1+s},
\end{align*}
where we used that $A$ is abelian.   Hence $(\mathrm{W})_{1,3}$,
$(\mathrm{W})_{1,4}$ and the three conditions of type~(C) hold
identically, $(\mathrm{P})_{1,3}$ holds since $\beta^2 \equiv 1 \mod{M}$, and
$(\mathrm{W})_{1,2}$, $(\mathrm{P})_{1,2}$ and $(\mathrm{P})_{1,4}$ are
equivalent to \eqref{eq:C5}--\eqref{eq:C6}, to
\eqref{eq:C3}--\eqref{eq:C4} and to \eqref{eq:C2}, respectively.

Summarising, all the conditions of Lemma~\ref{lem:consistency}
for~\eqref{eq:reduced-pres} hold if and only if
\eqref{eq:C1}--\eqref{eq:C6} hold.   Note that the conditions
$(\mathrm{W})_{1,4}$ and $(\mathrm{W})_{2,3}$ hold identically because
$x^w \in \langle x\rangle$ and $z^y \in \langle z\rangle$, which is a
consequence of $r = b = 0$; in the general case of
Section~\ref{sec:nontrivial} they give genuine conditions.

\subsection{Completing the proof}

By Lemma~\ref{lem:consistency} and the verification carried out
above, the presentation \eqref{eq:reduced-pres} is consistent if and only if the
congruences \eqref{eq:C1}--\eqref{eq:C6} hold, and in this case it
defines a group
$G$ of order exactly $\rho_1 \rho_2 \rho_3 \rho_4 = 4NM$, in which every
element is represented by exactly one normal word.

We now deduce the structure of~$G$.   By
Lemma~\ref{lem:consistency}, $|G| = 4NM$ and distinct normal words
represent distinct elements of~$G$.   Let $H_0$ be the set of elements
represented by the normal words $y^{\varepsilon} x^i$, with
$\varepsilon \in \{0, 1\}$ and $0 \le i < N$.   Since
$y^{\varepsilon} x^i \cdot y^{\varepsilon'} x^{i'}
= y^{\varepsilon + \varepsilon'} x^{\alpha^{\varepsilon'} i + i'}$ by
the relations $y^2 = 1$ and $x^y = x^\alpha$, the set $H_0$ is closed
under multiplication, hence it is a subgroup of the finite group~$G$.
It contains $x$ and $y$, so $H_0 = \langle x, y\rangle = H$, and
$|H| = 2N$ by the uniqueness of normal words.   The group
$\SD_{2^n}$ of~\eqref{eq:SD-def} maps onto $H$, and both groups have
order $2N$, so $H \cong \SD_{2^n}$.   In the same way $K$ is the set of
elements represented by the normal words $w^{\delta} z^j$, and
$K \cong \SD_{2^m}$, using $w^2 = 1$ and $z^w = z^\beta$.   An element
of $H \cap K$ is represented both by a normal word $y^{\varepsilon} x^i$
and by a normal word $w^{\delta} z^j$, so all exponents vanish by the
uniqueness of normal words, and $H \cap K = \{1\}$.   Therefore
$|HK| = |H|\cdot|K| = 4NM = |G|$, and $G = HK$ is an exact product.
Note that this argument uses the full strength of
Lemma~\ref{lem:consistency}, namely that the presentation defines a
group of order exactly $4NM$, and not only the validity of the
relations in some quotient.

Finally, $\langle x\rangle$ and $\langle z\rangle$ are normal in
$G$, since by the relations $x^y = x^\alpha$, $x^w = x^{1+s}$,
$z^y = z^{1+a}$, $z^w = z^\beta$ and $[x,z] = 1$ they are invariant
under conjugation by the generators of~$G$. \qed

\section{Counter-examples to ``$s=0$ when $n=m$''}\label{sec:examples}

Theorem~\ref{thm:main} parameterises the exact products $G = HK$ in
terms of the four-tuple $(a, s, t, c)$ subject to the six
congruences~\eqref{eq:C1}--\eqref{eq:C6}.   When the two factors have
the same order, i.e.\ $n = m$, the symmetry between $H$ and $ K$
might naively suggest a corresponding symmetry on the parameters,
forcing in particular the value $s = 0$ in every valid tuple.   
The purpose of this section is to
disprove this expectation. We show that the value $s = 2^{n-2}-2$
occurs in valid exact products even when $n = m$,
and that, in the case $n = m = 4$, the majority of the admissible
parameter tuples have $s \ne 0$.   This justifies the fact that the formulation of
Theorem~\ref{thm:main} cannot be tightened along this apparently
natural direction.

\begin{proposition}\label{prop:counter-s}
For $n = m = 4$, the tuple $(a, s, t, c) = (0, 2, 0, 0)$ defines an
exact product $G = HK$ with $H, K\cong\SD_{16}$, $|G| = 256$,
$\langle x\rangle,\langle z\rangle\trianglelefteq G$, and the
non-trivial value $s = 2$.
\end{proposition}

\begin{proof}
The tuple satisfies the system~\eqref{eq:C1}--\eqref{eq:C6}, as
the following computations show.
\begin{itemize}
  \item \eqref{eq:C1} holds, since $(1+0)^2 = 1 \equiv 1\mod{8}$.
  \item \eqref{eq:C2} holds, since $(1+2)^2 = 9\equiv 1\mod{8}$.
  \item \eqref{eq:C3} holds, since $0\cdot(2+2) = 0$.
  \item \eqref{eq:C4} holds, since $0\cdot(1+3) = 0$.
  \item \eqref{eq:C5} holds, since $0\cdot(1+3) = 0$.
  \item \eqref{eq:C6} holds, since $0\cdot(2+0) = 0$.
\end{itemize}
By Theorem~\ref{thm:main}, the corresponding presentation defines a
valid exact product.   Independently, the GAP~\cite{GAP} script of
Appendix~\ref{app:gap-s5} shows by coset enumeration that the
presentation defines a group $G$ of order $256$, and it checks in the
regular permutation representation of $G$ that $|H| = |K| = 16$,
$H \cong K \cong \SD_{16}$, $H\cap K = \{1\}$ and
$\langle x\rangle,\langle z\rangle\trianglelefteq G$.
\end{proof}

\begin{remark}\label{rem:s-counts}
For $n = m = 4$, the script of Appendix~\ref{app:gap-s5} determines
the admissible tuples $(a, s, t, c) \in \Z_8^4$ in two independent
ways.   First, it solves the congruences
\eqref{eq:C1}--\eqref{eq:C6}.   Second, for each of the $8^4$ tuples it
constructs the group defined by~\eqref{eq:reduced-pres} by coset
enumeration, which terminates for every tuple, and it checks directly
whether this group is an exact product of order $256$ of two copies of
$\SD_{16}$ with $\langle x\rangle$ and $\langle z\rangle$ normal.   The
two methods give the same $144$ tuples, in accordance with
Theorem~\ref{thm:main}.   Of these $144$ tuples, exactly
$24$ have $s = 0$, and the remaining $120$ have $s\in\{2, 4, 6\}$,
distributed as follows.
\[
  s = 0 \text{ in } 24 \text{ tuples },\quad s = 2 \text{ in } 48 \text{ tuples },\quad s = 4 \text{ in } 24 \text{ tuples },\quad s = 6 \text{ in }  48 \text{ tuples }.
\]
Thus, in the case $n = m = 4$, the tuples with $s \ne 0$ account
for $120$ of the $144$ admissible tuples.   We do not know whether the
tuples with $s \ne 0$ are the majority for every $n = m \ge 4$.   The
same $144$ tuples also appear in Subsection~\ref{sec:data44}, as the tuples
with $[x,z] = 1$ and $n_1 = m_1 = 1$.
\end{remark}

\section{Generalisation to arbitrary cores}\label{sec:nontrivial}

We now extend Theorem~\ref{thm:main} to the case in which the cores
of $\langle x\rangle$ and $\langle z\rangle$ in~$G$ are arbitrary
subgroups
\[
  \langle x\rangle_G = \langle x^{n_1}\rangle, \qquad
  \langle z\rangle_G = \langle z^{m_1}\rangle,
\]
where $n_1 \mid N$ and $m_1 \mid M$ are powers of~$2$, say $n_1 = 2^\nu$
($0 \le \nu \le n-1$) and $m_1 = 2^\mu$ ($0 \le \mu \le m-1$).      The
case of Theorem~\ref{thm:main} corresponds to
$n_1 = m_1 = 1$, i.e.\ $\nu = \mu = 0$.

Throughout this section we restrict to presentations in which
$[x, z] = 1$.   The case $[x, z] \ne 1$ is discussed in
Subsection~\ref{sec:xz-nontrivial}.

\begin{theorem}\label{thm:gen}
Let $n, m \ge 4$, and let $n_1 = 2^\nu \mid N$, $m_1 = 2^\mu \mid M$.
The presentation
\begin{equation}\label{eq:pres-gen}
  G = \left\langle x, y, z, w\,\middle|\,
  \begin{array}{l}
    x^N = y^2 = z^M = w^2 = 1,\ [x,z] = 1, \\
    x^y = x^\alpha,\ z^w = z^\beta, \\
    {[z,y]} = x^r z^a,\
    {[x,w]} = x^s z^b,\
    {[y,w]} = x^t z^c
  \end{array}
  \right\rangle
\end{equation}
defines an exact product $G = HK$ with $H = \langle x\rangle \rtimes\langle y\rangle \cong \SD_{2^n}$,
$K = \langle z\rangle \rtimes \langle w\rangle \cong \SD_{2^m}$, $H \cap K = \{1\}$, and
\begin{equation}\label{eq:cores}
  \langle x\rangle_G = \langle x^{n_1}\rangle,
  \qquad
  \langle z\rangle_G = \langle z^{m_1}\rangle,
\end{equation}
if and only if the six parameters $(r, a, s, b, t, c) \in
\Z_N \times \Z_M \times \Z_N \times \Z_M \times \Z_N \times \Z_M$ satisfy
the following conditions.
\begin{enumerate}[label=\textnormal{(\roman*)}]
  \item the twelve congruences
  \begin{align}
    r\,(\alpha + 1 + a) &\equiv 0 \mod{N}, \tag{D1}\label{eq:D1}\\
    (1+a)^2 &\equiv 1 \mod{M}, \tag{D2}\label{eq:D2}\\
    (1+s)^2 &\equiv 1 \mod{N}, \tag{D3}\label{eq:D3}\\
    b\,(1+s+\beta) &\equiv 0 \mod{M}, \tag{D4}\label{eq:D4}\\
    r\,(\beta - 1 - s) &\equiv 0 \mod{N}, \tag{D5}\label{eq:D5}\\
    b\,r &\equiv 0 \mod{M}, \tag{D6}\label{eq:D6}\\
    b\,(\alpha - 1 - a) &\equiv 0 \mod{M}, \tag{D7}\label{eq:D7}\\
    r\,b &\equiv 0 \mod{N}, \tag{D8}\label{eq:D8}\\
    t\,(2 + s) &\equiv 0 \mod{N}, \tag{D9}\label{eq:D9}\\
    c\,(1+\beta) + t\,b &\equiv 0 \mod{M}, \tag{D10}\label{eq:D10}\\
    t\,(1 + \alpha) + c\,r &\equiv 0 \mod{N}, \tag{D11}\label{eq:D11}\\
    c\,(2 + a) &\equiv 0 \mod{M}; \tag{D12}\label{eq:D12}
  \end{align}
  \item the additive order of $r$ in $\Z_N$ equals $m_1$;
  \item the additive order of $b$ in $\Z_M$ equals $n_1$.
\end{enumerate}
Conversely, let $G = HK$ be an exact product with a distinguished
generating system $(x, y, z, w)$ in the sense of
Definition~\ref{def:distinguished}, such that $[x,z] = 1$ and the cores
are as in~\eqref{eq:cores}.   Then its parameter tuple is of the form
$(0, 0, r, a, s, b, t, c)$ with $(r, a, s, b, t, c)$ satisfying
(i)--(iii), and $G$ is isomorphic to the group defined
by~\eqref{eq:pres-gen} via an isomorphism mapping generators to
generators.   Different tuples may give isomorphic exact products, see
Remark~\ref{rem:what-classified}.
Note that conditions~(ii) and~(iii) force $m_1 \mid N$ and
$n_1 \mid M$, since the additive order of an element of $\Z_N$ divides
$N$ and that of an element of $\Z_M$ divides~$M$.   Hence an exact
product as above can exist only if $n_1$ and $m_1$ divide both $N$ and
$M$, that is, only if $n_1, m_1 \mid 2^{\min(n,m)-1}$, see also
Lemma~\ref{lem:order-rb}.   Not all such pairs $(n_1, m_1)$ occur, as
the case $n = m = 4$ in Subsection~\ref{sec:data44} shows.
\end{theorem}

\begin{remark}\label{rem:thm-main-cor}
Theorem~\ref{thm:main} is the special case $n_1 = m_1 = 1$ of
Theorem~\ref{thm:gen}.   Indeed, conditions~(ii)--(iii) then force
$r = b = 0$, whereupon \eqref{eq:D1} and
\eqref{eq:D4}--\eqref{eq:D8} become trivial, and the remaining
six conditions \eqref{eq:D2}, \eqref{eq:D3}, \eqref{eq:D9},
\eqref{eq:D10}, \eqref{eq:D11}, \eqref{eq:D12} reduce to
\eqref{eq:C1}--\eqref{eq:C6}.
\end{remark}

\begin{remark}\label{rem:mixed-terms}
	The terms $tb$ in \eqref{eq:D10} and $cr$ in \eqref{eq:D11} arise
	from the mixed contributions of $r$ and $b$ in the conjugation rules
	$z^y = x^r z^{1+a}$ and $x^w = x^{1+s} z^b$, and were absent in the
	analogous conditions \eqref{eq:C4} and \eqref{eq:C5} of
	Theorem~\ref{thm:main} due to the hypothesis $r = b = 0$.
\end{remark}

\subsection{Necessity of the order conditions}

The next lemma gives the order conditions~(ii) and~(iii) of
Theorem~\ref{thm:gen}, and it is the semidihedral analogue of the
corresponding conditions in the dihedral classification
of~\cite{HuYu2025}.

Throughout this subsection, $G = HK$ is an exact product with a
distinguished generating system $(x, y, z, w)$ such that $[x,z] = 1$,
with parameter tuple $(0, 0, r, a, s, b, t, c)$.   Recall that the core
of a subgroup $U$ of $G$ is $U_G = \bigcap_{g \in G} U^g$.   It is a
normal subgroup of $G$ contained in $U$, and it contains every normal
subgroup $V$ of $G$ with $V \le U$, since $V = V^g \le U^g$ for every
$g \in G$.   Hence $U_G$ is the largest normal subgroup of $G$
contained in~$U$.

\begin{lemma}\label{lem:order-rb}
In the setting above, the following hold.
\begin{enumerate}[label=\normalfont(\alph*)]
  \item For $d \mid N$, the subgroup $\langle x^d\rangle$ is normal in
    $G$ if and only if $d b \equiv 0 \mod{M}$.   For $d \mid M$, the
    subgroup $\langle z^d\rangle$ is normal in $G$ if and only if
    $d r \equiv 0 \mod{N}$.
  \item $\ord_{\Z_M}(b)$ divides both $M$ and $N$, and
    $\ord_{\Z_N}(r)$ divides both $N$ and $M$.
  \item $\langle x\rangle_G = \langle x^{\ord_{\Z_M}(b)}\rangle$ and
    $\langle z\rangle_G = \langle z^{\ord_{\Z_N}(r)}\rangle$.
\end{enumerate}
In particular, for $n_1 \mid N$ and $m_1 \mid M$ we have
$\langle x\rangle_G = \langle x^{n_1}\rangle$ if and only if $b$ has
additive order exactly $n_1$ in $\Z_M$, and
$\langle z\rangle_G = \langle z^{m_1}\rangle$ if and only if $r$ has
additive order exactly $m_1$ in~$\Z_N$.
\end{lemma}

\begin{proof}
We prove the statements about $x$ and $b$, and the statements
about $z$ and $r$ are proved in the same way, interchanging the roles
of $(x, w, b)$ and $(z, y, r)$.   Recall that every element of
$\langle x\rangle\langle z\rangle$ can be written uniquely as
$x^i z^j$ with $i \in \Z_N$ and $j \in \Z_M$.

(a) Since $G$ is finite, a subgroup is normal in $G$ as soon as it
is mapped into itself by conjugation by each of the generators
$x, y, z, w$.   Conjugation by $x$ and by $z$ fixes $x^d$, since
$[x,z] = 1$, and $(x^d)^y = x^{\alpha d} \in \langle x^d\rangle$.
Finally,
\[
  (x^d)^w = (x^{1+s} z^b)^d = x^{d(1+s)}\, z^{d b},
\]
and by the uniqueness of the expression $x^i z^j$ this element lies in
$\langle x^d\rangle$ if and only if $z^{db} = 1$, that is,
$d b \equiv 0 \mod{M}$.

(b) The additive order of $b$ in $\Z_M$ divides~$M$.   Moreover,
$1 = (x^N)^w = (x^{1+s} z^b)^N = x^{(1+s)N} z^{bN} = z^{bN}$, so
$b N \equiv 0 \mod{M}$, and the additive order of $b$ divides~$N$.

(c) Let $o = \ord_{\Z_M}(b)$.   The subgroups of $\langle x\rangle$
are the subgroups $\langle x^d\rangle$ with $d \mid N$, and for
$d, d' \mid N$ we have $\langle x^d\rangle \le \langle x^{d'}\rangle$ if
and only if $d' \mid d$.   By~(a), $\langle x^d\rangle$ is normal in $G$
if and only if $o \mid d$.   By~(b), $o \mid N$, so
$\langle x^o\rangle$ is a normal subgroup of $G$ contained in
$\langle x\rangle$, and it contains every normal subgroup of $G$
contained in $\langle x\rangle$.   By the maximality of the core,
$\langle x\rangle_G = \langle x^o\rangle$.

The last assertion follows from~(c), since for $d, d' \mid N$ we
have $\langle x^d\rangle = \langle x^{d'}\rangle$ if and only if
$d = d'$.
\end{proof}

\subsection{Proof of Theorem~\ref{thm:gen}}

First we prove the necessary condition.  Conditions
\eqref{eq:D1}--\eqref{eq:D12} arise from the same four identities
$(\mathcal{I}_1)$--$(\mathcal{I}_4)$ of the proof of
Theorem~\ref{thm:main}, together with two additional identities
$(\mathcal{I}_3')$ and $(\mathcal{I}_4')$, namely
\[
(z^w)^y = (z^y)^w
\qquad\text{and}\qquad
(x^y)^w = (x^w)^y.
\]
These hold in $G$ because, under our hypothesis $[x,z] = 1$, the
subgroup $A = \langle x, z\rangle$ is abelian, and a direct
computation shows that both $z$ and $x$ commute with the mixed
commutator $[y, w] = x^t z^c \in A$, that is,
\[
  [x, [y,w]] = [z, [y,w]] = 1.
\]
Equivalently, $z^{wy} = z^{yw}$ and $x^{wy} = x^{yw}$.
In the case treated in Theorem~\ref{thm:main}, with
$\langle x\rangle, \langle z\rangle \trianglelefteq G$, one has
$r = b = 0$ and computing each side of the two identities in normal
form yields the same expression trivially; here, where $r$ and $b$
are allowed to be non-zero, computing each side via the conjugation
rules~\eqref{eq:mixed} and equating the $x$- and $z$-components
yields the four congruences \eqref{eq:D5}--\eqref{eq:D8}.   All
verifications below are direct algebraic expansions in the abelian
group $A = \langle x, z\rangle$.

We compute the two new identities; the others are unchanged from
Section~\ref{sec:main-thm} apart from the appearance of non-zero $r$
and $b$ in the cross-terms.

\paragraph{Identity $(\mathcal{I}_3')$: $(z^w)^y = (z^y)^w$.}
We compute the two sides separately by iterating the conjugation
rules~\eqref{eq:mixed}, and we obtain
\begin{align*}
  (z^w)^y &= (z^\beta)^y = (z^y)^\beta = (x^r z^{1+a})^\beta
        = x^{r\beta}\,z^{\beta(1+a)}, \\
  (z^y)^w &= (x^r z^{1+a})^w = (x^w)^r\,(z^w)^{1+a}
        = (x^{1+s} z^b)^r\,z^{\beta(1+a)}
        = x^{r(1+s)}\,z^{br + \beta(1+a)}.
\end{align*}
Equating the $x$- and $z$-components yields \eqref{eq:D5} and
\eqref{eq:D6} respectively.

\paragraph{Identity $(\mathcal{I}_4')$: $(x^y)^w = (x^w)^y$.}
Similarly, we have
\begin{align*}
  (x^y)^w &= (x^\alpha)^w = (x^w)^\alpha = (x^{1+s} z^b)^\alpha
        = x^{\alpha(1+s)}\,z^{\alpha b}, \\
  (x^w)^y &= (x^{1+s} z^b)^y = (x^y)^{1+s}\,(z^y)^b
        = x^{\alpha(1+s)}\,(x^r z^{1+a})^b
        = x^{\alpha(1+s)+rb}\,z^{(1+a)b}.
\end{align*}
Equating yields \eqref{eq:D7} and \eqref{eq:D8}.

\paragraph{Identity $(\mathcal{I}_1)$ revisited.}
We had $(z^y)^y = x^{r(\alpha+1+a)}\,z^{(1+a)^2}$, equating to $z$.
With $r$ possibly non-zero, the $x$-component is no longer
automatically zero, so we recover both \eqref{eq:D1} and \eqref{eq:D2}.

\paragraph{Identity $(\mathcal{I}_2)$ revisited.}
Similarly $(x^w)^w = x^{(1+s)^2}\,z^{b(1+s+\beta)}$, yielding
\eqref{eq:D3} and \eqref{eq:D4}.

\paragraph{Identities $(\mathcal{I}_3)$ and $(\mathcal{I}_4)$ revisited.}
We compute the identities $(\mathcal{I}_3)$ and $(\mathcal{I}_4)$
of the proof of Theorem~\ref{thm:main} without assuming $r = b = 0$.   Since $[x,z] = 1$, the subgroup
$A = \langle x\rangle\langle z\rangle$ is abelian, and every element of
$A$ can be written uniquely as $x^i z^j$ with $i \in \Z_N$ and
$j \in \Z_M$, because $\langle x\rangle \cap \langle z\rangle
\subseteq H \cap K = \{1\}$.   In particular $x^i z^j = 1$ if and only
if $i \equiv 0 \mod{N}$ and $j \equiv 0 \mod{M}$.   By~\eqref{eq:mixed},
conjugation by $y$ and by $w$ maps $A$ into itself, and for
$i \in \Z_N$, $j \in \Z_M$ we have
\begin{equation}\label{eq:conj-on-A}
  (x^i z^j)^y = (x^\alpha)^i (x^r z^{1+a})^j
  = x^{\alpha i + r j}\, z^{(1+a) j},
  \qquad
  (x^i z^j)^w = (x^{1+s} z^b)^i (z^\beta)^j
  = x^{(1+s) i}\, z^{b i + \beta j}.
\end{equation}
Recall that $y^w = y\,[y,w] = y\,x^t z^c$.

For $(\mathcal{I}_3)$, the relation $w^2 = 1$ gives
$y = (y^w)^w$, and by~\eqref{eq:conj-on-A}
\[
  (y^w)^w = (y\,x^t z^c)^w = y^w\,(x^t z^c)^w
  = y\,x^t z^c \cdot x^{(1+s)t}\, z^{bt + \beta c}
  = y\,x^{(2+s)t}\, z^{(1+\beta)c + tb}.
\]
Hence $x^{(2+s)t}\, z^{(1+\beta)c + tb} = 1$, which gives
\[
  t\,(2 + s) \equiv 0 \mod{N}
  \qquad\text{and}\qquad
  c\,(1+\beta) + t\,b \equiv 0 \mod{M},
\]
that is, \eqref{eq:D9} and~\eqref{eq:D10}.

For $(\mathcal{I}_4)$, the relation $y^2 = 1$ gives
$(y^w)^2 = (y^2)^w = 1$, and by~\eqref{eq:conj-on-A}
\[
  (y^w)^2 = y\,x^t z^c\, y\,x^t z^c
  = y^2\,(x^t z^c)^y\, x^t z^c
  = x^{\alpha t + r c}\, z^{(1+a)c} \cdot x^t z^c
  = x^{(1+\alpha)t + c r}\, z^{(2+a)c}.
\]
Hence
\[
  t\,(1 + \alpha) + c\,r \equiv 0 \mod{N}
  \qquad\text{and}\qquad
  c\,(2 + a) \equiv 0 \mod{M},
\]
that is, \eqref{eq:D11} and~\eqref{eq:D12}.   The terms $tb$ and $cr$
come from the $z$-component $z^{bt}$ of $(x^t)^w$ and from the
$x$-component $x^{rc}$ of $(z^c)^y$, respectively, and they vanish when
$r = b = 0$ (see Remark~\ref{rem:mixed-terms}).   The same computations
appear, as consistency conditions $(\mathrm{P})_{1,2}$ and
$(\mathrm{W})_{1,2}$, in the sufficiency part below.

The necessity of conditions~(ii) and~(iii) is exactly the content of
Lemma~\ref{lem:order-rb}.

Finally we prove the sufficient condition.   We apply
Lemmas~\ref{lem:consistency} and~\ref{lem:reformulation} to the
presentation~\eqref{eq:pres-gen}, with the same ordering
$g_1 = w$, $g_2 = y$, $g_3 = z$, $g_4 = x$ of the generators, relative
orders $(\rho_1, \rho_2, \rho_3, \rho_4) = (2, 2, M, N)$ and trivial
power relations as in the proof of
Theorem~\ref{thm:main}.   The conjugate relations, written as normal
words, are now
\begin{equation}\label{eq:conj-relations-gen}
\begin{aligned}
  g_2^{g_1} &= y^w = g_2\,g_3^c\,g_4^t,
  &g_3^{g_1} &= z^w = g_3^\beta,
  &g_4^{g_1} &= x^w = g_3^b\,g_4^{1+s},\\
  g_3^{g_2} &= z^y = g_3^{1+a}\,g_4^r,
  &g_4^{g_2} &= x^y = g_4^\alpha,
  &g_4^{g_3} &= x^z = g_4.
\end{aligned}
\end{equation}
We list all twenty conditions, again in decreasing order of the
smallest index.   We point out in advance that the order
conditions~(ii) and~(iii) of the theorem, as opposed to the
congruences \eqref{eq:D1}--\eqref{eq:D12}, are used in exactly three
places, namely for the conditions $(\mathrm{W})_{2,3}$ and
$(\mathrm{W})_{1,4}$, and at the end of the proof to determine the
cores.   The four conditions of family~(4) hold automatically
by Lemma~\ref{lem:reformulation}.

\emph{Smallest index $i = 3$.}   The sub-presentation on $z, x$ is
the same as in the proof of Theorem~\ref{thm:main}, so
$(\mathrm{W})_{3,4}$ and $(\mathrm{P})_{3,4}$ hold and
$A = \langle z, x\rangle \cong \Z_M \times \Z_N$.   Conjugation by $y$
and by $w$ now acts on words in $A$ by
\[
  (x^i z^j)^y = x^{\alpha i + r j} z^{(1+a) j},
  \qquad
  (x^i z^j)^w = x^{(1+s) i} z^{b i + \beta j}.
\]

\emph{Smallest index $i = 2$.}   We have
\begin{align*}
  (\mathrm{W})_{2,3}&\colon\quad (x^r z^{1+a})^M = x^{rM} z^{(1+a)M}
  = x^{rM},
  \qquad
  (\mathrm{W})_{2,4}\colon\quad (x^\alpha)^N = 1,\\
  (\mathrm{P})_{2,3}&\colon\quad z \mapsto x^r z^{1+a} \mapsto
  (x^r z^{1+a})^y = x^{\alpha r}\,(x^r z^{1+a})^{1+a}
  = x^{r(\alpha + 1 + a)}\, z^{(1+a)^2},\\
  (\mathrm{P})_{2,4}&\colon\quad x \mapsto x^\alpha \mapsto x^{\alpha^2},
\end{align*}
and $(\mathrm{C})_{2,3,4}$ reads $(x^z)^y = x^\alpha$ on the left and
$(x^y)^{z^y} = (x^\alpha)^{x^r z^{1+a}} = x^\alpha$ on the right, since
$A$ is abelian.   Thus $(\mathrm{W})_{2,4}$ and $(\mathrm{C})_{2,3,4}$
hold identically, $(\mathrm{P})_{2,4}$ holds since $\alpha^2 \equiv 1 \mod{N}$,
$(\mathrm{P})_{2,3}$ is equivalent to \eqref{eq:D1}--\eqref{eq:D2}, and
$(\mathrm{W})_{2,3}$ is equivalent to $rM \equiv 0 \mod{N}$.   The last
congruence follows from condition~(ii), since $m_1 r \equiv 0 \mod{N}$
and $m_1 \mid M$.

\emph{Smallest index $i = 1$.}   The words $\gamma_{1j}$ are
$y^w = y\,x^t z^c$, $z^w = z^\beta$ and $x^w = x^{1+s} z^b$.   The
conditions of type~(W) are
\begin{align*}
  (\mathrm{W})_{1,2}&\colon\quad
  (y\,x^t z^c)^2 = y^2\,(x^t z^c)^y\,x^t z^c
  = x^{\alpha t + r c} z^{(1+a)c}\,x^t z^c
  = x^{(1+\alpha) t + c r}\, z^{(2+a) c},\\
  (\mathrm{W})_{1,3}&\colon\quad (z^\beta)^M = 1,
  \qquad
  (\mathrm{W})_{1,4}\colon\quad (x^{1+s} z^b)^N = x^{(1+s)N} z^{bN}
  = z^{bN},
\end{align*}
where $(\mathrm{W})_{1,2}$ is computed in the group defined by the
sub-presentation on $y, z, x$.   The conditions of type~(P) are
\begin{align*}
  (\mathrm{P})_{1,2}&\colon\quad
  y \mapsto y\,x^t z^c \mapsto (y\,x^t z^c)^w
  = y\,x^t z^c\,(x^t z^c)^w\\
  &\phantom{\colon\quad y \mapsto y\,x^t z^c \mapsto (y\,x^t z^c)^w}
  = y\,x^t z^c\,x^{(1+s)t} z^{bt + \beta c}
  = y\,x^{(2+s)t}\, z^{(1+\beta)c + tb},\\
  (\mathrm{P})_{1,3}&\colon\quad z \mapsto z^\beta \mapsto z^{\beta^2},\\
  (\mathrm{P})_{1,4}&\colon\quad x \mapsto x^{1+s} z^b \mapsto
  (x^{1+s} z^b)^w = (x^{1+s} z^b)^{1+s} z^{\beta b}
  = x^{(1+s)^2}\, z^{b(1+s+\beta)}.
\end{align*}
Finally, the conditions of type~(C) are as follows, where on the right
we use that conjugation by $y^w = y\,x^t z^c$ acts on $A$ as conjugation
by~$y$, since $A$ is abelian.
\begin{align*}
  (\mathrm{C})_{1,2,3}&\colon\quad
  (z^y)^w = (x^r z^{1+a})^w = x^{(1+s)r} z^{br + \beta(1+a)},
  &(z^w)^{y^w} &= (z^\beta)^y = x^{r\beta} z^{(1+a)\beta},\\
  (\mathrm{C})_{1,2,4}&\colon\quad
  (x^y)^w = (x^\alpha)^w = x^{\alpha(1+s)} z^{\alpha b},
  &(x^w)^{y^w} &= (x^{1+s} z^b)^y = x^{\alpha(1+s) + rb} z^{(1+a)b},\\
  (\mathrm{C})_{1,3,4}&\colon\quad
  (x^z)^w = x^{1+s} z^b,
  &(x^w)^{z^w} &= (x^{1+s} z^b)^{z^\beta} = x^{1+s} z^b,
\end{align*}
where in each line the left-hand side and the right-hand side of the
condition are computed separately.
Comparing exponents in $A \cong \Z_M \times \Z_N$, we obtain the
following.   $(\mathrm{W})_{1,3}$ and $(\mathrm{C})_{1,3,4}$ hold
identically and $(\mathrm{P})_{1,3}$ holds since $\beta^2 \equiv 1 \mod{M}$.
Moreover, $(\mathrm{W})_{1,2}$ is equivalent to
\eqref{eq:D11}--\eqref{eq:D12}, $(\mathrm{P})_{1,2}$ is equivalent to
\eqref{eq:D9}--\eqref{eq:D10}, $(\mathrm{P})_{1,4}$ is equivalent to
\eqref{eq:D3}--\eqref{eq:D4}, $(\mathrm{C})_{1,2,3}$ is equivalent to
\eqref{eq:D5}--\eqref{eq:D6}, and $(\mathrm{C})_{1,2,4}$ is equivalent
to \eqref{eq:D7}--\eqref{eq:D8}.
$(\mathrm{W})_{1,4}$ is equivalent to $bN \equiv 0 \mod{M}$, which
follows from condition~(iii), since $n_1 b \equiv 0 \mod{M}$ and
$n_1 \mid N$.

Summarising, all the conditions of Lemma~\ref{lem:consistency}
for~\eqref{eq:pres-gen} hold if and only if
\eqref{eq:D1}--\eqref{eq:D12} hold together with $rM \equiv 0 \mod{N}$
and $bN \equiv 0 \mod{M}$, and the last two congruences are consequences
of conditions~(ii) and~(iii).   Hence, under the hypotheses of the
theorem, the presentation defines a group $G$ of order exactly $4NM$ in
which distinct normal words represent distinct elements.   The argument
at the end of the proof of Theorem~\ref{thm:main} applies verbatim,
since it only uses the relations $x^N = y^2 = 1$, $x^y = x^\alpha$,
$z^M = w^2 = 1$, $z^w = z^\beta$ and the uniqueness of normal words.
It shows that $H = \langle x, y\rangle \cong \SD_{2^n}$,
$K = \langle z, w\rangle \cong \SD_{2^m}$, $H \cap K = \{1\}$ and
$G = HK$.   Finally, the cores are as in~\eqref{eq:cores} by
Lemma~\ref{lem:order-rb} applied to conditions~(ii) and~(iii). \qed

\subsection{The case \texorpdfstring{\boldmath$[x, z] \ne 1$}{[x,z] != 1}}\label{sec:xz-nontrivial}

By Lemma~\ref{lem:trivial-cores}, the commutator $[x,z]$ is trivial
when both $\langle x\rangle$ and $\langle z\rangle$ are normal in~$G$.
In general $[x,z]$ need not be trivial, and the following lemma records
what can be said when one of the two subgroups is normal.   Throughout
this subsection, $G = HK$ is an exact product with a distinguished
generating system $(x, y, z, w)$ in the sense of
Definition~\ref{def:distinguished}, and, as in~\eqref{eq:mixed}, we
write $[x,z] = x^{e_1} z^{e_2}$ with $e_1 \in \Z_N$ and $e_2 \in \Z_M$.

\begin{lemma}\label{lem:xz-normal}
If $\langle x\rangle \trianglelefteq G$, then $e_2 \equiv 0 \mod{M}$.
Symmetrically, if $\langle z\rangle \trianglelefteq G$, then
$e_1 \equiv 0 \mod{N}$.   In particular, if $[x,z] \ne 1$, then at
least one of $\langle x\rangle$ and $\langle z\rangle$ is not normal
in~$G$.
\end{lemma}

\begin{proof}
Since $\langle x\rangle \cap \langle z\rangle \subseteq H \cap K = \{1\}$,
every element of $\langle x\rangle\langle z\rangle$ can be written
uniquely as $x^i z^j$ with $i \in \Z_N$ and $j \in \Z_M$.   If
$\langle x\rangle \trianglelefteq G$, then
$[x,z] = x^{-1} x^z \in \langle x\rangle$, and uniqueness gives
$z^{e_2} = 1$, that is $e_2 \equiv 0 \mod{M}$.   If
$\langle z\rangle \trianglelefteq G$, then
$[x,z] = (z^{-1})^x z \in \langle z\rangle$, and uniqueness gives
$e_1 \equiv 0 \mod{N}$.
Finally, if both $\langle x\rangle$ and $\langle z\rangle$ are
normal in~$G$, then $e_1 \equiv 0 \mod{N}$ and $e_2 \equiv 0 \mod{M}$,
that is, $[x,z] = 1$, as in Lemma~\ref{lem:trivial-cores}.   This
proves that if $[x,z] \ne 1$, then at least one of $\langle x\rangle$
and $\langle z\rangle$ is not normal in~$G$.
\end{proof}

One might expect the stronger localisation
$[x,z] \in \langle x^{n_1}\rangle \langle z^{m_1}\rangle$, that is
$e_1 \equiv 0 \mod{n_1}$ and $e_2 \equiv 0 \mod{m_1}$.   This is false
in general, see Subsection~\ref{sec:data44}.

In~\cite{HuYu2025}, Hu and Yu classify the exact products of two
dihedral groups $D_{2n}$ and $D_{2m}$ with $n$ and $m$ odd, and in this
setting they prove that $[x,z] = 1$, see Remark~\ref{rem:dihedral}
below.   We now show that in the semidihedral setting non-trivial
commutators $[x, z]$ do occur.

\begin{theorem}\label{thm:xz-nontrivial-example}
There exists an exact product $G = HK$ with $H = \langle x \rangle \rtimes \langle y\rangle
\cong \SD_{16}$, $K = \langle z \rangle \rtimes \langle w\rangle \cong \SD_{16}$, both cores
non-trivial, and $[x, z] \ne 1$.

Specifically, the presentation
\begin{equation}\label{eq:xz-counter}
  G = \left\langle x, y, z, w \,\middle|\,
  \begin{array}{@{}l@{\qquad}l@{\qquad}l@{}}
    x^8 = 1,\ \ y^2 = 1,\ \ x^y = x^3, & {[x,z]} = x^2 z^2, & {[z,y]} = x^4,\\[2pt]
    z^8 = 1,\ \ w^2 = 1,\ \ z^w = z^3, & {[x,w]} = z^4,     & {[y,w]} = 1
  \end{array}
  \,\right\rangle
\end{equation}
defines an exact product of order~$256$ with
$\langle x\rangle_G = \langle x^2\rangle$, $\langle z\rangle_G = \langle z^2\rangle$
(both of order~$4$).
In the notation of~\eqref{eq:mixed} and of Theorems~A and~B, this
exact product corresponds to $n = m = 4$, $N = M = 8$,
$n_1 = m_1 = 2$ and
\[
  (e_1, e_2, r, a, s, b, t, c) = (2, 2, 4, 0, 0, 4, 0, 0)
  \in \Z_8^8 .
\]
\end{theorem}

\begin{proof}
The computations are carried out with GAP~\cite{GAP}, by means of
the script reproduced in Appendix~\ref{app:gap}, whose output is also
reported there.   First, Todd--Coxeter coset enumeration with respect
to the trivial subgroup (see~\cite[Chapter~5]{Sims}
or~\cite[Chapter~5]{HoltEickOBrien2005}) shows that the
presentation~\eqref{eq:xz-counter} defines a group $G$ of order
exactly~$256$.   This is a statement about the finitely
presented group itself, and it is stronger than the existence of
a permutation representation of degree $256$ in which the relations
hold.   The coset table then gives the regular permutation
representation of $G$ on its $256$ elements, which is faithful, and the
remaining checks are performed in this representation.   They show the
following.
\begin{itemize}
  \item The elements $x$ and $z$ have order $8$, and $y$ and $w$ have
    order~$2$, and the relations $x^y = x^3$ and $z^w = z^3$ hold.
  \item $|H| = |K| = 16$, the groups $H$ and $K$ are both isomorphic to
    $\SD_{16}$, and $|H \cap K| = 1$, so that $G = HK$ is an exact
    product.
  \item The cores are $\langle x\rangle_G = \langle x^2\rangle$ and
    $\langle z\rangle_G = \langle z^2\rangle$, computed as the
    intersections of all the conjugates of $\langle x\rangle$ and of
    $\langle z\rangle$.
  \item $[x, z] = x^2 z^2 \ne 1$.
\end{itemize}
This completes the proof.
\end{proof}

\paragraph{\upshape\bfseries Direct verification of the cores.}
Once it is known that $G = HK$ is an exact product of order
$256$, the cores can also be determined by a direct computation, which
we now give.   Since $\langle x\rangle \cap \langle z\rangle \subseteq
H \cap K = \{1\}$, every element of $\langle x\rangle\langle z\rangle$
can be written uniquely as $x^i z^j$ with $i, j \in \Z_8$.   By the
defining relations, the conjugates of $x$ and $z$ by the four
generators are
\[
  \begin{array}{llll}
    x^x = x, & x^y = x^3, & x^z = x\,[x,z] = x^3 z^2, & x^w = x\,[x,w] = x z^4,\\[2pt]
    z^z = z, & z^w = z^3, & z^x = z\,[x,z]^{-1} = z^{-1} x^{-2}, & z^y = z\,[z,y] = z x^4.
  \end{array}
\]
In particular $x^z \notin \langle x\rangle$, since $z^2 \ne 1$, and
$z^x \notin \langle z\rangle$, since otherwise
$x^{-2} \in \langle x\rangle \cap \langle z\rangle = \{1\}$.   Hence
neither $\langle x\rangle$ nor $\langle z\rangle$ is normal in~$G$, and
the cores are contained in the unique maximal subgroups
$\langle x^2\rangle$ and $\langle z^2\rangle$ of these cyclic
$2$-groups.

It remains to show that $\langle x^2\rangle$ and
$\langle z^2\rangle$ are normal, that is, that they are mapped into
themselves by conjugation by the four generators.
Directly from the
relations, we obtain $(x^2)^x = x^2$, $(x^2)^y = x^6$, $(z^2)^z = z^2$ and
$(z^2)^w = z^6$.   The remaining four conjugates are consequences of the
single identity
\begin{equation}\label{eq:key-identity}
  (x^2)^z = x^{-2},
\end{equation}
as we now show.   Since conjugation by $z$ maps $x^{2}$ to $x^{-2}$, so does
conjugation by $z^{-1}$, which gives $x^{-2} z^{-1} = z^{-1} x^{2}$.
Hence
\[
  (z^2)^x = (z^{-1} x^{-2})^2 = z^{-1}\,(x^{-2} z^{-1})\,x^{-2}
  = z^{-1} z^{-1} x^{2} x^{-2} = z^{-2}.
\]
Squaring~\eqref{eq:key-identity} and the last identity gives
$(x^4)^z = x^{-4} = x^4$ and $(z^4)^x = z^{-4} = z^4$, so that $x^4$
commutes with $z$ and $z^4$ commutes with~$x$.   Therefore
\[
  (x^2)^w = (x z^4)^2 = x^2 z^8 = x^2
  \qquad\text{and}\qquad
  (z^2)^y = (z x^4)^2 = z^2 x^8 = z^2.
\]
Thus $\langle x^2\rangle$ and $\langle z^2\rangle$ are normal in $G$,
and $\langle x\rangle_G = \langle x^2\rangle$,
$\langle z\rangle_G = \langle z^2\rangle$.

The identity~\eqref{eq:key-identity} holds in $G$, as the script
of Appendix~\ref{app:gap} confirms, but we have not found a short
derivation of it by hand from the defining relations.   Since the
presentation~\eqref{eq:xz-counter} is not covered by
Theorem~\ref{thm:gen}, the multiplication in $\langle x, z\rangle$ is
not given by a polycyclic normal form, and~\eqref{eq:key-identity}
depends on the consistency of the whole presentation, which is
established by the coset enumeration.   Thus the computation above
reduces the determination of the cores to the single
identity~\eqref{eq:key-identity}, but it does not make the example
independent of the computer.

\begin{remark}\label{rem:xz-example-params}
We compare the exact product $G$ of
Theorem~\ref{thm:xz-nontrivial-example} with those of
Theorem~\ref{thm:gen}, in order to see which features of the case
$[x,z] = 1$ survive when $[x,z] \ne 1$.   On the one hand, $G$ is very
close to the setting of Theorem~\ref{thm:gen}.   Indeed, the tuple
$(r, a, s, b, t, c) = (4, 0, 0, 4, 0, 0)$ satisfies the
congruences \eqref{eq:D1}--\eqref{eq:D12} with $\alpha = \beta = 3$,
and $r$ and $b$ both have additive order $2 = n_1 = m_1$ in $\Z_8$.
Hence, by Theorem~\ref{thm:gen}, the same tuple with $[x,z] = 1$
defines an exact product with the same cores, and $G$ is obtained from
it by replacing the relation $[x,z] = 1$ with $[x,z] = x^2 z^2$.   On
the other hand, the subgroup $\langle x, z\rangle$ behaves differently.
In the proof of Theorem~\ref{thm:gen} the subgroup
$A = \langle x\rangle\langle z\rangle$ is abelian, and conjugation by
$y$ and by $w$ acts on it by the formulas~\eqref{eq:conj-on-A}, on which
the computations of the proof rest.   In $G$ the subgroup
$\langle x, z\rangle$ has order $64$, so that it coincides with
$\langle x\rangle\langle z\rangle$, but it has nilpotency class~$3$, and
$[x,z] = x^2 z^2$ is not central in it, as a direct computation with
GAP~\cite{GAP} shows.   Thus the method of
Sections~\ref{sec:main-thm} and~\ref{sec:nontrivial} does not apply
directly when $[x,z] \ne 1$, and Theorem~\ref{thm:gen} does not exhaust
all exact products of two semidihedral groups.   Moreover,
Subsection~\ref{sec:data44} shows that most exact products with
$[x,z] \ne 1$ are not obtained from tuples of Theorem~\ref{thm:gen} by
changing $(e_1, e_2)$ only.   The classification of these exact
products, with $(e_1, e_2)$ as additional parameters, is left open, see
Section~\ref{sec:open}.
\end{remark}

\begin{remark}\label{rem:dihedral}
Theorem~\ref{thm:xz-nontrivial-example} has no counterpart in the
dihedral classification of~\cite{HuYu2025}, where $[x,z] = 1$ always
holds.   We explain where the argument of~\cite{HuYu2025} uses that $n$
and $m$ are odd, since this shows why it cannot be transferred to
semidihedral groups.   In the notation of the present paper, let $G = HK$ be an exact product with
$H = \langle x\rangle \rtimes \langle y\rangle \cong D_{2n}$ and
$K = \langle z\rangle \rtimes \langle w\rangle \cong D_{2m}$, so that
$x^y = x^{-1}$ and $z^w = z^{-1}$, where $n, m \ge 3$ are odd, and let
$\langle z\rangle_G = \langle z^{m_1}\rangle$ with $m_1 \mid m$.   In the
proof of~\cite[Theorem~1]{HuYu2025}, Hu and Yu show that $x$ centralises
$z_1 = z^{m_1}$, and that $[x,z] = z_1^{e}$ and $z_1^y = z_1^{u}$ for
some $e, u \in \Z_{m/m_1}$.   Comparing different expressions for the
same elements of $G$, they obtain congruences in $e$, $u$ and the other
parameters which imply $2e \equiv 0 \mod{m/m_1}$.   Since $m$ is odd,
$2$ is invertible modulo $m/m_1$, so that $e \equiv 0 \mod{m/m_1}$, that
is, $[x,z] = 1$.   Thus the dihedral analogue of
Theorem~\ref{thm:xz-nontrivial-example} does not hold for dihedral
groups $D_{2n}$ and $D_{2m}$ with $n$ and $m$ odd.

In the semidihedral setting the orders $N$ and $M$ are powers of
$2$, and a cancellation of this kind is not available.   Indeed,
Theorem~\ref{thm:xz-nontrivial-example} and Subsection~\ref{sec:data44} show
that $[x,z]$ can be non-trivial, and that no comparable restriction on
$(e_1, e_2)$ holds.   We have not investigated whether non-trivial
commutators $[x,z]$ also occur for exact products of two dihedral
groups of $2$-power order, which are not covered by~\cite{HuYu2025}.
\end{remark}

\subsection{The case \texorpdfstring{\boldmath$n = m = 4$}{n = m = 4}}\label{sec:data44}
To give evidence on the regime $[x,z] \ne 1$, we determined with
GAP~\cite{GAP} all tuples
$(e_1, e_2, r, a, s, b, t, c) \in \Z_8^8$ for which the presentation
\[
  \left\langle x, y, z, w \,\middle|\,
  \begin{array}{l}
    x^8 = y^2 = z^8 = w^2 = 1,\ x^y = x^3,\ z^w = z^3,\\
    {[x,z]} = x^{e_1} z^{e_2},\
    {[z,y]} = x^r z^a,\
    {[x,w]} = x^s z^b,\
    {[y,w]} = x^t z^c
  \end{array}
  \right\rangle
\]
defines a group $G$ of order $256$ in which $x, z$ have order $8$,
$y, w$ have order $2$, $|H| = |K| = 16$ and $H \cap K = \{1\}$.   This
is the same convention as in Theorems~\ref{thm:main}
and~\ref{thm:gen}.   For each of the $8^8$ tuples we first computed by
coset enumeration the index of $K = \langle z, w\rangle$, and for the
tuples giving index~$16$ we constructed the regular permutation
representation of~$G$, checked the conditions above and computed the
cores of $\langle x\rangle$ and $\langle z\rangle$.   The tuples for
which the enumeration of the cosets of $K$ did not terminate within a
fixed bound were treated separately, as explained below.

There are $14785$ such tuples.   Exactly $880$ of them have
$[x,z] = 1$, and these are precisely the tuples
$(0, 0, r, a, s, b, t, c)$ with $(r, a, s, b, t, c)$ satisfying
\eqref{eq:D1}--\eqref{eq:D12} and the order conditions of
Theorem~\ref{thm:gen}, in accordance with that theorem.   The
remaining $13905$ tuples have $[x,z] \ne 1$, and their distribution
according to the cores
$\langle x\rangle_G = \langle x^{n_1}\rangle$ and
$\langle z\rangle_G = \langle z^{m_1}\rangle$ is the following.
\begin{center}
\small
\begin{tabular}{crrl}
\toprule
$(n_1, m_1)$ & $[x,z] = 1$ & $[x,z] \ne 1$ & values of $(e_1, e_2)$ with $[x,z] \ne 1$\\
\midrule
$(1,1)$ & $144$ & $0$ & ---\\
$(1,2)$ & $144$ & $1056$ & $(2,0), (4,0), (6,0)$\\
$(2,1)$ & $144$ & $1056$ & $(0,2), (0,4), (0,6)$\\
$(1,4)$ & $64$ & $128$ & $(2,0), (4,0)$\\
$(4,1)$ & $64$ & $128$ & $(0,4), (0,6)$\\
$(1,8)$ & $16$ & $16$ & $(4,0)$\\
$(8,1)$ & $16$ & $16$ & $(0,4)$\\
$(2,2)$ & $160$ & $8827$ & all $(e_1, e_2) \in (2\Z_8)^2 \setminus \{(0,0)\}$\\
$(2,4)$ & $64$ & $1338$ & $(0,4), (2,2), (2,6), (4,0), (4,4), (6,2), (6,6)$\\
$(4,2)$ & $64$ & $1340$ & $(0,4), (2,2), (2,6), (4,0), (4,4), (6,2), (6,6)$\\
\midrule
total & $880$ & $13905$ & \\
\bottomrule
\end{tabular}
\end{center}
For reasons of space, the table only gives, for each pair
$(n_1, m_1)$, the number of tuples and the values of $(e_1, e_2)$ that
occur.   The complete list of the $14785$ tuples was determined with
GAP, and its availability is described in the data availability
statement at the end of the paper.

Some features of these data deserve comment.   First, $e_1$ and $e_2$
are always even.   Second, in accordance with
Lemma~\ref{lem:xz-normal}, one has $e_2 = 0$ whenever $n_1 = 1$ and
$e_1 = 0$ whenever $m_1 = 1$, and $[x,z]$ can be non-trivial even when
one of $\langle x\rangle$ and $\langle z\rangle$ is normal in~$G$.
Third, the localisation $[x,z] \in \langle x^{n_1}\rangle\langle z^{m_1}\rangle$
fails for $2043$ tuples.   For instance, the tuple
$(e_1, e_2, r, a, s, b, t, c) = (2, 2, 0, 2, 0, 2, 0, 0)$ defines an
exact product with $\langle x\rangle_G = \langle x^4\rangle$ and
$\langle z\rangle_G = \langle z^2\rangle$, whereas
$[x,z] = x^2 z^2 \notin \langle x^4\rangle\langle z^2\rangle$.
Fourth, only $2832$ of the $13905$ tuples have the property that
$(r, a, s, b, t, c)$ is admissible for Theorem~\ref{thm:gen} with the
same values of $n_1$ and $m_1$.   Finally, in a random sample of $400$
of these groups, $\langle x, z\rangle$ always has order $64$, and its
nilpotency class is $2$ in $73$ cases and $3$ in $327$ cases.
Class~$3$ occurs already for the exact product of
Theorem~\ref{thm:xz-nontrivial-example}, see
Remark~\ref{rem:xz-example-params}, and for instance also for the tuple
$(0, 6, 4, 2, 2, 6, 0, 6)$.

The tuples for which the coset enumeration did not terminate within
the bound were analysed by computing the largest $2$-quotient of class
at most~$8$ of the finitely presented group, which has every quotient
of order~$256$ as a quotient, and by a coset enumeration with a much
larger bound.
There are $1038$ such tuples.   For $1032$ of them this $2$-quotient has
order different from $256$.   Since a group of order $256$ has
nilpotency class at most~$7$, it coincides with its largest $2$-quotient
of class at most~$8$, and therefore for these $1032$ tuples the
presentation does not define a group of order~$256$.   The script used
for this analysis is also covered by the data availability statement.
Among the remaining six tuples, for two the presentation defines a group of
order $256$ with the required properties, and these are included in
the counts above.   For four further tuples, among them
$(2, 4, 0, 6, 4, 0, 0, 0)$, this $2$-quotient is an exact product of
order $256$ with the required properties, while the coset enumeration
of the finitely presented group did not terminate.   These four tuples
are therefore realised by an exact product, but we could not establish
that the presentation itself defines a group of order~$256$, and they
are not included in the counts above.   More generally, a tuple may be realised by an exact product which
is a proper quotient of the group defined by the presentation, and
apart from these four tuples and the two tuples discussed below we
have not investigated this possibility.   This phenomenon also explains why the rows $(2,4)$ and
$(4,2)$ of the table contain different numbers of tuples.   Exchanging
the roles of $(x, y)$ and $(z, w)$ maps the exact products with cores
$\langle x^4\rangle$ and $\langle z^2\rangle$ to those with cores
$\langle x^2\rangle$ and $\langle z^4\rangle$.   For $1338$ of the
$1340$ tuples of the row $(4,2)$ the new parameter tuple is again in
the list.   For the remaining two, namely $(4, 4, 0, 2, 4, 6, 0, 4)$ and
$(6, 2, 0, 2, 4, 6, 0, 4)$, the new tuples are
$(4, 4, 6, 4, 2, 0, 4, 0)$ and $(6, 2, 6, 4, 2, 0, 4, 0)$, and the
corresponding presentations define groups of order $2304$, whose
largest $2$-quotients are the exact products in question.

Finally,  these data concern presentations of the form
displayed above.   We have not attempted to determine whether every
exact product of two copies of $\SD_{16}$ admits generators for which
all mixed commutators lie in $\langle x\rangle\langle z\rangle$.

\section{Conclusions and open problem}\label{sec:open}

In this paper we initiated the classification of Zappa--Sz\'ep
products $G = HK$ of two semidihedral groups $H \cong \SD_{2^n}$,
$K \cong \SD_{2^m}$ with $n, m \ge 4$.   Our main contributions are
twofold.   First, under the normality hypothesis
$\langle x\rangle \trianglelefteq G$ and
$\langle z\rangle \trianglelefteq G$, we proved
(Theorem~\ref{thm:main}) that the mixed-commutator parameters
satisfy $[x,z] = 1$, $r = 0$ and $b = 0$ identically, and that the
admissible parameter tuples $(a, s, t, c)$ are exactly the solutions
of the explicit system of six congruences~\eqref{eq:C1}--\eqref{eq:C6}.
Second, dropping the normality hypothesis and assuming only that the
cores are $\langle x\rangle_G = \langle x^{n_1}\rangle$ and
$\langle z\rangle_G = \langle z^{m_1}\rangle$ for arbitrary
$n_1 \mid N$ and $m_1 \mid M$, under the simplifying assumption
$[x,z] = 1$, and for exact products with a distinguished
generating system, we obtained (Theorem~\ref{thm:gen}) a complete
description of the admissible parameter tuples $(r, a, s, b, t, c)$,
by twelve congruences together with order conditions on $r$ and $b$.
In both cases the exact product is determined by its parameter tuple,
while the question of when two tuples give isomorphic exact products is
not addressed, see Remark~\ref{rem:what-classified}.   Theorem~\ref{thm:main} arises as
the specialisation of Theorem~\ref{thm:gen} to the case
$n_1 = m_1 = 1$.

Theorem~\ref{thm:main} was complemented by explicit
examples (Section~\ref{sec:examples}), showing that the
parameter $s = 0$ is not forced in the equal-rank case $n = m$.
Finally, in Subsection~\ref{sec:xz-nontrivial} we exhibited an
explicit exact product with both cores non-trivial and
$[x,z] \ne 1$ (Theorem~\ref{thm:xz-nontrivial-example}), a
phenomenon that does not occur in the classification of Hu and
Yu~\cite{HuYu2025} for dihedral groups $D_{2n}$ and $D_{2m}$ with $n$
and $m$ odd, see
Remark~\ref{rem:dihedral}.

The most natural problem left open by this work is the full extension
of Theorem~\ref{thm:gen} to the regime $[x,z] \ne 1$.
Lemma~\ref{lem:xz-normal} and the data of Subsection~\ref{sec:data44} show
that this regime is considerably richer than the one treated in
Theorem~\ref{thm:gen}.   For $n = m = 4$ it contains $13905$ parameter
tuples, against $880$ with $[x,z] = 1$, and most of them are not
obtained from admissible tuples of Theorem~\ref{thm:gen} by a simple
modification of $(e_1, e_2)$.   Moreover, the subgroup
$\langle x, z\rangle$ is no longer abelian and its nilpotency class
may be $3$, and the presentations involved are no longer polycyclic
with respect to the ordering of the generators used in
Section~\ref{sec:main-thm}.   A description of this regime by
polynomial congruences, if it exists, will require new ideas.

A different approach to problems of this kind, based on
structural methods of abstract group theory rather than on the
consistency theorem for polycyclic presentations, has been developed
by Du, Luo and Yu~\cite{DuLuoYu2024}, who characterise the products
$X = GC$ of a generalised quaternion or dihedral group $G$ and a
cyclic group $C$ with $G \cap C = 1$ by a reduction to the structure
of the cores and of suitable quotients of~$X$.   It would be
interesting to investigate whether these methods can also be applied
to the regime $[x,z] \ne 1$.

\section*{Acknowledgements}
The author acknowledges the funding support from MUR-Italy via PRIN 2022RFAZCJ ``Algebraic methods in Cryptanalysis''. 
The author is grateful to the two anonymous reviewers for their
careful reading of the manuscript and for their valuable comments and
suggestions.   The detailed report of the first reviewer led to a
substantial improvement in the precision of the statements and in the
rigour of the proofs, in particular of the consistency arguments and of
the computational verifications.   The suggestion of the second
reviewer to determine all the examples with $[x,z] \ne 1$ in the case
$n = m = 4$ led to Subsection~\ref{sec:data44} and to a better understanding
of this case.

\section*{Statements and Declarations}
The author has no relevant financial or non-financial interests to disclose.

Data availability statement: the GAP scripts used for the
computations of Section~\ref{sec:examples} and of
Theorem~\ref{thm:xz-nontrivial-example} are reproduced in
Appendix~\ref{app:gap-main}.   The data generated in
Subsection~\ref{sec:data44}, namely the list of the parameter tuples
described there, together with the GAP code used to obtain and analyse
them, are available from the author upon reasonable request.   All the
computations were carried out with GAP~4.12.1~\cite{GAP}.

\appendix
\section{GAP code}\label{app:gap-main}
All the computations of this appendix were carried out with
GAP~4.12.1~\cite{GAP}.   The scripts below, as well as the scripts and
the data of Subsection~\ref{sec:data44}, are also available from the
author in electronic form upon request.

\subsection{The enumeration of Section~\ref{sec:examples}}\label{app:gap-s5}

The following script determines the admissible tuples
$(a, s, t, c)$ of Theorem~\ref{thm:main} for $n = m = 4$ in the two
independent ways described in Remark~\ref{rem:s-counts}, compares the
results and computes the distribution of~$s$.   It also verifies
Proposition~\ref{prop:counter-s}, since $(0, 2, 0, 0)$ is one of the
tuples found.

{\small
\begin{verbatim}
# Verification of the enumeration of Section 5 (Remark 5.2), n = m = 4.
# Tested with GAP 4.12.1.
# The admissible tuples (a,s,t,c) of Theorem 4.1 are determined in two
# independent ways, by solving the congruences (C1)-(C6) and by constructing
# the group defined by the presentation of Theorem 4.1, for every
# tuple in Z_8^4.
N := 8;; M := 8;; al := 3;; be := 3;;
F := FreeGroup("x", "y", "z", "w");;
x := F.1;; y := F.2;; z := F.3;; w := F.4;;
# 1. solutions of the congruences (C1)-(C6)
Cong := function(a, s, t, c)
  return ((1+a)^2 - 1) mod M = 0 and ((1+s)^2 - 1) mod N = 0
     and (t*(2+s)) mod N = 0 and (c*(1+be)) mod M = 0
     and (t*(1+al)) mod N = 0 and (c*(2+a)) mod M = 0;
end;;
T := Cartesian([0..M-1], [0..N-1], [0..N-1], [0..M-1]);;
S1 := Filtered(T, v -> CallFuncList(Cong, v));;
# 2. tuples for which this presentation defines an exact product of
#    order 4NM with <x> and <z> normal, checked directly on the group
IsAdmissible := function(v)
  local a, s, t, c, rels, ct, gx, gy, gz, gw, G, H, K;
  a := v[1]; s := v[2]; t := v[3]; c := v[4];
  rels := [ x^N, y^2, z^M, w^2, Comm(x, z), x^y / x^al, z^w / z^be,
            Comm(z, y) / z^a, Comm(x, w) / x^s, Comm(y, w) / (x^t * z^c) ];
  # coset enumeration over the trivial subgroup; it fails if the group
  # defined by the presentation has more than 10^5 elements
  ct := CosetTableFromGensAndRels([x, y, z, w], rels, [] :
                                  max := 100000, silent := true);
  if ct = fail or Length(ct[1]) <> 4*N*M then return false; fi;
  gx := PermList(ct[1]);; gy := PermList(ct[3]);;
  gz := PermList(ct[5]);; gw := PermList(ct[7]);;
  G := Group(gx, gy, gz, gw);; H := Group(gx, gy);; K := Group(gz, gw);;
  return Size(H) = 2*N and Size(K) = 2*M and Size(Intersection(H, K)) = 1
     and IdGroup(H) = [16, 8] and IdGroup(K) = [16, 8]
     and IsNormal(G, Group(gx)) and IsNormal(G, Group(gz));
end;;
S2 := Filtered(T, IsAdmissible);;
Print([Length(S1), Length(S2), S1 = S2], "\n");
# distribution of s among the admissible tuples
Print(Collected(List(S1, v -> v[2])), "\n");
\end{verbatim}
}

Running the script produces the following output.

{\small
\begin{verbatim}
[ 144, 144, true ]
[ [ 0, 24 ], [ 2, 48 ], [ 4, 24 ], [ 6, 48 ] ]
\end{verbatim}
}

The first line says that both methods give $144$ tuples, and that
the two sets coincide.   The second line gives the number of admissible
tuples for each value of~$s$.   We also checked that the coset
enumeration terminates within the prescribed bound for all the $4096$
tuples, so that the second method decides every tuple.

\subsection{Theorem~\ref{thm:xz-nontrivial-example}}\label{app:gap}

The following GAP~\cite{GAP} script verifies the assertions of
Theorem~\ref{thm:xz-nontrivial-example}, following the items listed in
its proof, together with the identity~\eqref{eq:key-identity} and the
structure of $\langle x, z\rangle$ described in
Remark~\ref{rem:xz-example-params}.

{\small
\begin{verbatim}
# Verification of Theorem 6.6 (exact product of two copies of SD_16
# with [x,z] = x^2 z^2 different from 1).
# Tested with GAP 4.12.1.
F := FreeGroup("x", "y", "z", "w");;
x := F.1;; y := F.2;; z := F.3;; w := F.4;;
rels := [ x^8, y^2, z^8, w^2,
          x^y / x^3, z^w / z^3,
          Comm(x, z) / (x^2 * z^2),
          Comm(z, y) / x^4,
          Comm(x, w) / z^4,
          Comm(y, w) ];;
G := F / rels;;
# 1-2. order of the group defined by the presentation (coset enumeration)
Print(Size(G), "\n");
# regular permutation representation, used for the remaining checks
iso := IsomorphismPermGroup(G);;  P := Image(iso);;
gx := Image(iso, G.1);;  gy := Image(iso, G.2);;
gz := Image(iso, G.3);;  gw := Image(iso, G.4);;
# 3. orders of x, y, z, w
Print(List([gx, gy, gz, gw], Order), "\n");
# 4. semidihedral relations
Print([gx^gy = gx^3, gz^gw = gz^3], "\n");
# 5. orders of H, K and of their intersection, and isomorphism type
H := Subgroup(P, [gx, gy]);;  K := Subgroup(P, [gz, gw]);;
Print([Size(H), Size(K), Size(Intersection(H, K))], "\n");
Print([IdGroup(H), IdGroup(K), IdGroup(SmallGroup(16, 8))], "\n");
# 6. cores of <x> and <z>
Print([Core(P, Subgroup(P, [gx])) = Subgroup(P, [gx^2]),
       Core(P, Subgroup(P, [gz])) = Subgroup(P, [gz^2])], "\n");
# 7. the commutator [x,z]
Print([Comm(gx, gz) = gx^2 * gz^2, IsOne(Comm(gx, gz))], "\n");
# the identity (x^2)^z = x^(-2) used in the direct verification
Print((gx^2)^gz = gx^-2, "\n");
# structure of <x,z>
A := Subgroup(P, [gx, gz]);;
Print([Size(A), NilpotencyClassOfGroup(A),
       IsCentral(A, Subgroup(P, [Comm(gx, gz)]))], "\n");
\end{verbatim}
}

Running the script with GAP~4.12.1 produces the following output,
line by line.

{\small
\begin{verbatim}
256
[ 8, 2, 8, 2 ]
[ true, true ]
[ 16, 16, 1 ]
[ [ 16, 8 ], [ 16, 8 ], [ 16, 8 ] ]
[ true, true ]
[ true, false ]
true
[ 64, 3, false ]
\end{verbatim}
}

Here \texttt{SmallGroup(16, 8)} is the semidihedral group of
order $16$ in the SmallGroups library of GAP.   The first line is the
order of the group defined by the presentation, computed by coset
enumeration, and all the other lines are computed in its regular
permutation representation.


\begin{thebibliography}{99}

\bibitem{BurnessLi2021}
T.~C.\ Burness and C.~H.\ Li,
\textit{On solvable factors of almost simple groups},
Adv.\ Math.\ \textbf{377} (2021), 107499.

\bibitem{Douglas1961}
J.\ Douglas,
\textit{On the supersolvability of bicyclic groups},
Proc.\ Natl.\ Acad.\ Sci.\ USA \textbf{47} (1961), 1493--1495.

\bibitem{DuLuoYu2024}
S.~F.\ Du, W.~J.\ Luo and H.\ Yu,
\textit{The product of a generalized quaternion group and a cyclic group},
J.\ Aust.\ Math.\ Soc.\ \textbf{118} (2025), no.~1, 31--64.

\bibitem{Eick2000}
B.\ Eick,
\textit{Algorithms for polycyclic groups},
Habilitationsschrift, Universit\"at Kassel, 2001.

\bibitem{GAP}
The GAP~Group,
\textit{GAP -- Groups, Algorithms, and Programming,
Version 4.12.1}, 2022.
\url{https://www.gap-system.org}.

\bibitem{HeringLiebeckSaxl1987}
C.\ Hering, M.~W.\ Liebeck and J.\ Saxl,
\textit{The factorizations of the finite exceptional groups of Lie type},
J.\ Algebra \textbf{106} (1987), no.~2, 517--527.

\bibitem{HoltEickOBrien2005}
D.~F.\ Holt, B.\ Eick and E.~A.\ O'Brien,
\textit{Handbook of Computational Group Theory},
Discrete Mathematics and its Applications,
Chapman \& Hall/CRC, Boca Raton, FL, 2005.

\bibitem{HKK2024}
K.\ Hu, I.\ Kov\'acs and Y.\ S.\ Kwon,
\textit{On exact products of a dihedral group and a cyclic group},
Comm.\ Algebra \textbf{53} (2025), no.~2, 854--874.

\bibitem{HuYu2025}
K.\ Hu and H.\ Yu,
\textit{On exact products of two dihedral groups},
Comm.\ Algebra \textbf{53} (2025), no.~10, 4206--4214.

\bibitem{Ito1955}
N.\ It\^o,
\textit{\"Uber das Produkt von zwei abelschen Gruppen},
Math.\ Z.\ \textbf{62} (1955), 400--401.

\bibitem{Jones2002}
G.~A.\ Jones,
\textit{Cyclic regular subgroups of primitive permutation groups},
J.\ Group Theory \textbf{5} (2002), no.~4, 403--407.

\bibitem{Kegel1961}
O.\ H.\ Kegel,
\textit{Produkte nilpotenter Gruppen},
Arch.\ Math.\ (Basel) \textbf{12} (1961), 90--93.

\bibitem{LiWangXia2023}
C.\ H.\ Li, L.\ Wang and B.\ Xia,
\textit{The exact factorizations of almost simple groups},
J.\ London Math.\ Soc.\ (2) \textbf{108} (2023), 1417--1447.

\bibitem{LPS1996}
M.\ W.\ Liebeck, C.\ E.\ Praeger and J.\ Saxl,
\textit{On factorizations of almost simple groups},
J.\ Algebra \textbf{185} (1996), 409--419.

\bibitem{LiebeckPraegerSaxl2010}
M.~W.\ Liebeck, C.~E.\ Praeger and J.\ Saxl,
\textit{Regular subgroups of primitive permutation groups},
Mem.\ Amer.\ Math.\ Soc.\ \textbf{203} (2010), no.~952, vi+74~pp.

\bibitem{Miller1935}
G.\ A.\ Miller,
\textit{Groups which are the products of two permutable proper subgroups},
Proc.\ Natl.\ Acad.\ Sci.\ USA \textbf{21} (1935), no.~7, 469--472.

\bibitem{Ore1937}
O.\ Ore,
\textit{Structures and group theory.\ I.},
Duke Math.\ J.\ \textbf{3} (1937), 149--174.

\bibitem{Sims}
C.\ C.\ Sims,
\textit{Computation with Finitely Presented Groups},
Encyclopedia of Math.\ and its Applications \textbf{48},
Cambridge Univ.\ Press, 1994.

\bibitem{Szep1949}
J.\ Sz\'ep,
\textit{\"Uber die als Produkt zweier Untergruppen darstellbaren endlichen Gruppen},
Comment.\ Math.\ Helv.\ \textbf{22} (1949), 31--33.

\bibitem{Wielandt1951}
H.\ Wielandt,
\textit{\"Uber das Produkt von paarweise vertauschbaren nilpotenten Gruppen},
Math.\ Z.\ \textbf{55} (1951), 1--7.

\bibitem{Xia2017}
B.\ Xia,
\textit{Quasiprimitive groups containing a transitive alternating group},
J.\ Algebra \textbf{490} (2017), 555--567.

\bibitem{Yu2025}
H.\ Yu,
\textit{The exact product of a semi-dihedral group and a cyclic group},
J.\ Algebra Appl., Paper No.~2750022, published online
10 November 2025, DOI: 10.1142/S0219498827500228.

\bibitem{Zappa1940}
G.\ Zappa,
\textit{Sulla costruzione dei gruppi prodotto di due dati sottogruppi
permutabili tra loro},
in {\em Atti del Secondo Congresso dell'Unione Matematica Italiana,
Bologna, 1940}, pp.~119--125, Ed.\ Cremonese, Roma.

\end{thebibliography}
\end{document}